\documentclass[a4paper,11pt]{article}
\usepackage{amsmath,amssymb,enumerate}
\usepackage{amsthm,color}
\usepackage{txfonts}

\newcommand{\R}{\mathbb{R}}

\newcommand{\N}{\mathbb{N}}
\newcommand{\ep}{\varepsilon}
\newcommand{\pa}{\partial}

\DeclareMathOperator{\supp}{supp}

\newtheorem{theorem}{Theorem}[section]
\newtheorem{lemma}[theorem]{Lemma}
\newtheorem{proposition}[theorem]{Proposition}

\theoremstyle{remark}
\newtheorem{remark}{Remark}[section]
\theoremstyle{definition}

\newtheorem{definition}{Definition}[section]

\numberwithin{equation}{section}

\makeatletter
\def\@cite#1#2{[{{\bfseries #1}\if@tempswa , #2\fi}]}
\makeatother                  
\begin{document}
\begin{center}
\Large{{\bf
Test function method for blow-up phenomena
of semilinear wave equations and their weakly coupled systems%
}}
\end{center}

\vspace{5pt}

\begin{center}
Masahiro Ikeda%
\footnote{%
Department of Mathematics, Faculty of Science and Technology, Keio University, 3-14-1 Hiyoshi, Kohoku-ku, Yokohama, 223-8522, Japan/Center for Advanced Intelligence Project, RIKEN, Japan, 
E-mail:\ {\tt masahiro.ikeda@keio.jp/masahiro.ikeda@riken.jp}},
Motohiro Sobajima%
\footnote{
Department of Mathematics, 
Faculty of Science and Technology, Tokyo University of Science,  
2641 Yamazaki, Noda-shi, Chiba, 278-8510, Japan,  
E-mail:\ {\tt msobajima1984@gmail.com}}
\footnote{Partially supported 
by Grant-in-Aid for Young Scientists Research (B) 
No.18K13445. }
and
Kyouhei Wakasa%
\footnote{
Department of Mathematics, 
Faculty of Science and Technology, Tokyo University of Science,  
2641 Yamazaki, Noda-shi, Chiba, 278-8510, Japan,  
E-mail:\ {\tt wakasa\_kyouhei@ma.noda.tus.ac.jp}}
\end{center}

\newenvironment{summary}{\vspace{.5\baselineskip}\begin{list}{}{%
     \setlength{\baselineskip}{0.85\baselineskip}
     \setlength{\topsep}{0pt}
     \setlength{\leftmargin}{12mm}
     \setlength{\rightmargin}{12mm}
     \setlength{\listparindent}{0mm}
     \setlength{\itemindent}{\listparindent}
     \setlength{\parsep}{0pt}
     \item\relax}}{\end{list}\vspace{.5\baselineskip}}
\begin{summary}
{\footnotesize {\bf Abstract.}
In this paper we consider the wave equations with 
power type nonlinearities including 
time-derivatives of unknown functions and 
their weakly coupled systems. 
We propose a framework of test 
function method and give a simple proof 
of the derivation of sharp upper bound of lifespan 
of solutions to nonlinear wave equations and their systems. 
We point out that 
for respective critical case, we use 
a family of self-similar solution to the standard wave equation 
including Gauss's hypergeometric functions 
which are originally introduced by Zhou \cite{Zhou92}. 
However, our framework is much simpler than that. 
As a consequence, we found new $(p,q)$-curve for the system 
$\pa_t^2u-\Delta u=|v|^q$, $\pa_t^2v-\Delta v=|\pa_tu|^p$ 
and lifespan estimate for small solutions for new region. 
}
\end{summary}



\tableofcontents
\newpage
\section{Introduction}
\paragraph{}

In this paper we consider the semilinear wave equations 
with power type nonlinearities 
including derivatives of unknown functions 
and their weakly coupled systems 
\begin{equation}
\label{NW}
\begin{cases}
\pa_t^2u(x,t) -\Delta u(x,t)=G\big(u(x,t),\pa_tu(x,t)\big), 
&(x,t)\in \R^N\times (0,T),
\\
u (x,0)=\ep f(x)
&x\in \R^N, 
\\
\pa_tu (x,0)=\ep g(x)
&x\in \R^N.
\end{cases}
\end{equation}
and 
\begin{equation}
\label{NW-sys}
\begin{cases}
\pa_t^2u(x,t) -\Delta u(x,t)=G_1\big(v(x,t),\pa_tv(x,t)\big), 
&(x,t)\in \R^N\times (0,T),
\\
\pa_t^2v(x,t) -\Delta v(x,t)=G_2\big(u(x,t),\pa_tu(x,t)\big), 
&(x,t)\in \R^N\times (0,T),
\\
u(x,0)=\ep f_1(x), \quad v(x,0)=\ep f_2(x)
&x\in \R^N, 
\\
\pa_tu(x,0)=\ep g_1(x), \quad \pa_tv(x,0)=\ep g_2(x)
&x\in \R^N,
\end{cases}
\end{equation}
where $\pa_t=\pa/\pa t$, 
$\Delta=\sum_{j=1}^N\pa^2/\pa x_j^2$ and $T>0$. 
The nonlinear terms 
$G$, $G_1$ and $G_2$ are nonnegative 
and smooth (specified later) 
and $u$ and $v$ are unknown functions. 
Throughout this paper the initial values 
$(f,g)$, $(f_1,g_1)$ and $(f_2,g_2)$ 
are assumed to be satisfied the following condition 
\begin{equation}\label{ass.initial}
(f,g)\in C_c^\infty(\R^N), \quad I[g]:=\int_{\R^N}g(x)\,dx>0.
\end{equation}
Finally, the parameter $\ep>0$ describes the 
smallness of corresponding initial value. 
The aim of the present paper is 
to give a simple way to 
derive the corresponding sharp lifespan 
of blowup solutions to \eqref{NW} and \eqref{NW-sys} 
via a test function method.

The problem of blowup phenomena 
of \eqref{NW} has a long history. 
The study of this kind problem with $G(u)=|u|^p$
has been started by John \cite{John79}. 
He proved the following fact when $N=3$. 
\begin{itemize}
\item If $1<p<1+\sqrt{2}$, then 
the solution of \eqref{NW} blows up in finite time 
for ``positive'' initial value. 
\item If $p>1+\sqrt{2}$, then 
there exists a global solution with small initial value. 
\end{itemize}
After that Strauss \cite{Strauss81} 
conjectured that the threshold for dividing 
blowup phenomena in finite time for arbitrary ``positive''
small initial value and global existence of small solutions
is given by 
\[
p_S(N)=\sup\{p>1\;;\;\gamma_S(N,p)>0\}, 
\quad 
\gamma_S(N,p)=2+(N+1)p-(N-1)p^2.
\] 
Actucally, this gives $p_S(3)=1+\sqrt{2}$ as mentioned above and 
in the case $N=1$, Kato \cite{Kato80} proved 
blowup phenomena in finite time for arbitrary
``positive'' small initial value with $1<p<p_S(1)=\infty$. 
There are many subsequent papers dealing with 
the blowup phenomena. 
Then until the contributions of 
Yordanov--Zhang \cite{YZ06} and 
Zhou \cite{Zhou07}, the complete picture 
of blowup phenomena 
and existence of global solutions 
are clarified including the critical situation $p=p_S(N)$ 
(see also 
Glassey \cite{Glassey81a, Glassey81b}, 
Sideris \cite{Sideris84}, 
Schaeffer \cite{Schaeffer85}, 
Rammaha \cite{Rammaha89}, 
Georgiev--Lindblad--Sogge \cite{GLS97}).

The lifespan of blowup solutions to 
\eqref{NW} has been intensively considered. 
Here we refer 
Lindblad \cite{Lindblad90}, 
Zhou \cite{Zhou92,Zhou92-3,Zhou93}, 
Lindblad--Sogge \cite{LS96}, 
Di Pomponio--Georgiev \cite{DG01}, 
Takamura--Wakasa \cite{TW11} 
and 
Lai--Zhou \cite{LZ14}.
In view of the previous works listed above, 
the precise behavior of 
lifespan of small solutions 
with respect to the parameter $\ep>0$ sufficiently small: 
\[
{\rm LifeSpan}(u)\approx
\begin{cases}
C\ep^{-\frac{2p(p-1)}{\gamma_S(N,p)}}
&
{\rm if\ }1<p<p_S(N),
\\
\exp(C\ep^{-p(p-1)})
&
{\rm if\ }p=p_S(N). 
\end{cases}
\]
An alternative proof 
of lifespan estimate with critical case $p=p_S(N)$
via Gauss's hypergeometric function
can be found in 
Zhou \cite{Zhou93} and Zhou--Han \cite{ZH14}.

Similar problem can be found for \eqref{NW} with 
$G=|\pa_tu|^p$ 
(see e.g., 
John \cite{John81}, 
Sideris \cite{Sideris83}, 
Masuda \cite{Masuda83}, 
Schaeffer \cite{Schaeffer86}
Rammaha \cite{Rammaha87}, 
Agemi \cite{Agemi91}, 
Hidano--Tsutaya \cite{HT95}
Tzvetkov \cite{Tzvetkov98},
Zhou \cite{Zhou01}
and 
Hidano--Wang--Yokoyama \cite{HWY12}). 
The complete picture of the blowup phenomena 
for small solutions can be summarized as follows:
\[
{\rm LifeSpan}(u)\approx
\begin{cases}
C\ep^{-(\frac{1}{p-1}-\frac{N-1}{2})^{-1}}
&
{\rm if\ }1<p<\frac{N+1}{N-1},
\\
\exp(C\ep^{-(p-1)})
&
{\rm if\ }p=\frac{N+1}{N-1}. 
\\
\infty
&
{\rm if\ }\frac{N+1}{N-1}<p<\frac{N}{N-2}. 
\end{cases}
\]
We should remark that the global existence of 
small  solutions to \eqref{NW} with $G=|\pa_tu|^p$ 
is only proved under the initial value 
with radially symmetric in high spatial dimension.

The problem \eqref{NW} with 
combined type $G=|u|^q+|\pa_tu|^p$ 
has been recently discussed by 
Zhou--Han \cite{ZH14-comb}
and 
Hidano--Wang--Yokoyama \cite{HWY16}. 
In \cite{HWY16}, 
it is found the borderline of the position of $(p,q)$ 
for blowup phenomena of small solutions. 
Surprisingly, in the threshold case 
they proved the global existence of small solutions 
which is completely different from 
the situation of both the cases $G(u)=|u|^p$ and $G=|\pa_tu|^p$.  

In this connection, 
a similar interesting structure 
has been analysed for the weakly coupled problem 
\eqref{NW-sys}. 
In this case the interaction of 
each unknown functions $u$ and $v$ 
plays an important role. 
In particular, 
the situation depends heavily on 
the structure of nonlinear terms $G_1$ and $G_2$. 
This means that 
even in the special case $G_1=|v|^p$ and $G_2=|u|^q$, 
the position of $(p,q)$ (with describes the effect of 
nonlinearity) is quite important 
to discuss the behavior of solutions to this system. 
From this view point, 
many mathematicians try 
to find the blowup phenomena 
and global solutions of small solutions. 
Here we refer
Del Santo--Georgiev--Mitidieri \cite{DGM97}, 
Deng \cite{Deng97}, 
Del Santo--Mitidieri \cite{DM98}, 
Deng \cite{Deng99}, 
Kubo--Ohta \cite{KO99}, 
Agemi--Kurokawa--Takamura \cite{AKT00}, 
Kurokawa--Takamura \cite{KT03}, 
Kurokawa \cite{Kurokawa05}, 
Georgiev--Takamura--Zhou \cite{GTZ06}, 
Kurokawa--Takamura--Wakasa \cite{KTW12}
for the case $G_1=|v|^p$ and $G_2=|u|^q$, 
Deng \cite{Deng99}, 
Xu \cite{Xu04}, 
Kubo--Kubota--Sunagawa \cite{KKS06}
for the case $G_1=|\pa_tv|^p$ and $G_2=|\pa_t u|^q$, 
and 
Hidano--Yokoyama \cite{HY16}
for the case $G_1=|v|^q$ and $G_2=|\pa_tu|^p$.  

Recently in Ikeda--Sobajima \cite{IkedaSobajima3} 
an alternative test function method 
for nonlinear heat, Schr\"odinger, 
and damped wave equations 
has been introduced and 
the sharp upper bound of lifespan 
for respective equations are given.
Of course each equation 
has a huge mount of previous works 
(see e.g., Fujita \cite{Fujita66}, Hayakawa \cite{Hayakawa73}, 
Sugitani \cite{Sugitani75}, Kobayashi--Sirao--Tanaka \cite{KST77}, 
Li--Nee \cite{LN92} 
for the heat equations, 
Ikeda--Wakasugi \cite{IW13}, Fujiwara--Ozawa \cite{FO16}
for Schr\"odinger equations 
and 
Li--Zhou \cite{LZ95}, 
Lin--Nishihara--Zhai \cite{LNZ12}, Ikeda--Ogawa \cite{IO16}
Lai--Zhou \cite{LZarxiv}
for damped wave equations
the references therein).
Despite of this, 
the technique in \cite{IkedaSobajima3} 
gives us a short proof of sharp upper bound 
of lifespan of small solutions to respective equations 
and in some case, in particular the Schr\"odinger equation, the estimates given by this technique 
is not known. 
Moreover, it worth noticing that 
the initial value does not need to be required the positivity in the point-wise sense 
even in high dimensional cases $N\geq 4$.  
Therefore we expect that 
by introducing 
the technique in \cite{IkedaSobajima3}  
into the analysis of wave equations, 
one can give an alternative proof 
of sharp (for many cases) upper bound of lifespan of small solutions 
and the assumption on the initial value can be weaken. 

The first purpose of the present paper is
to propose a framework of test function method for nonlinear wave equation 
due to \cite{IkedaSobajima3} and give precise lifespan estimates 
for problems \eqref{NW} and \eqref{NW-sys} without 
the assumption of the positivity of initial value in the point-wise sense. 
The second is to find the new blowup region for the case \eqref{NW-sys} with 
$G_1=|v|^q$ and $G_2=|\pa_tu|^p$ and 
lifespan estimates for respective cases. 
Since in the present paper 
we focus our attention to 
the framework of test function method, 
we do not enter a discussion for 
existence of solutions to the respective problems. 
At this point, we refer 
Sideris \cite{Sideris84}, 
Kapitanskii \cite{Kapitanskii90},
Hidano--Wang--Yokoyama \cite{HWY12},
Georgiev--Takamura--Zhou \cite{GTZ06},
Kubo--Kubota--Sunagawa \cite{KKS06},
Hidano--Yokoyama \cite{HY16} and their references therein.

The present paper is organized as follows: 
In Section \ref{sec:Str-demo}, to explain our argument, 
we demonstrate the short derivation of the upper bound of lifespan 
for the special case $\pa_t^2u-\Delta u = |u|^p$
with $1<p<p_S(N)$. 
Section \ref{sec:Preliminaries} is devoted to 
describe the properties of super-solutions 
to the wave equations 
and self-similar solutions to the linear wave equation 
including Gauss's hypergeometric functions introduced in Zhou \cite{Zhou92}. 
Some useful lemmas indicating our test function method 
are stated and proved also in Section 3. 
The main results are stated at the beginning of each Sections 
\ref{sec:Str}, \ref{sec:Gla}, \ref{sec:combined}, 
\ref{sec:Str-Str}, \ref{sec:Gla-Gla} and \ref{sec:Str-Gla}. 
More precisely, in Section \ref{sec:Str}, 
we discuss 
\begin{equation*}
\pa_t^2u -\Delta u=|u|^p 
\quad
\text{in}\ \R^N\times (0,T)
\end{equation*}
for the critical case $p=p_S(N)$. 
Although the sharp lifespan estimate has been proved by Takamura--Wakasa \cite{TW11}
and an alternative proof was given by Zhou--Han \cite{ZH14}, 
we will give a (much) simpler proof. 
Then we discuss the equation 
\begin{equation*}
\pa_t^2u -\Delta u=|\pa_tu|^p 
\quad
\text{in}\ \R^N\times (0,T)
\end{equation*}
in Section \ref{sec:Gla}, which is related to Glassey conjecture. 
The eqaution with a combined type nonlinearity 
\begin{equation*}
\pa_t^2u -\Delta u=|u|^q+|\pa_tu|^p 
\quad
\text{in}\ \R^N\times (0,T)
\end{equation*}
will be dealt with in Section \ref{sec:combined}. After that 
the weakly coupled systems
\begin{equation*}
\begin{cases}
\pa_t^2u -\Delta u=a_{11}|v|^{p_{11}}+a_{12}|\pa_tv|^{p_{12}}, 
&
\text{in}\ \R^N\times (0,T),
\\
\pa_t^2v -\Delta v=a_{21}|u|^{p_{21}}+a_{22}|\pa_tu|^{p_{22}}, 
&
\text{in}\ \R^N\times (0,T),
\end{cases}
\end{equation*}
are considered 
when $a_{12}=a_{22}=0$ in Section \ref{sec:Str-Str},
when $a_{11}=a_{21}=0$ in Section \ref{sec:Gla-Gla} 
and 
when $a_{21}=a_{12}=0$ in Section \ref{sec:Str-Gla}, 
respectively. 
We point out that in Section \ref{sec:Str-Gla}, 
a new blowup position of $(p,q)$
is found and 
lifespan estimates including critical situations (on the critical curve) 
are derived.

\section{Alternative proof of blowup of $\pa_t^2u-\Delta u = |u|^p$ for $1<p<p_S(N)$}
\label{sec:Str-demo}

To begin with, we consider the following problem
\begin{equation}
\label{eq:Str-demo}
\begin{cases}
\pa_t^2u_\ep -\Delta u_\ep=|u_\ep|^p 
&
\text{in}\ \R^N\times (0,T),
\\
u_\ep(0)=\ep f
&
\text{in}\ \R^N, 
\\
\pa_t u_\ep(0)=\ep g
&
\text{in}\ \R^N, 
\end{cases}
\end{equation}
where we assume that $f$ and $g$
satisfies \eqref{ass.initial}. 
In this section we use 
\[
\gamma_S(N,p)=2+(N+1)p-(N-1)p^2, 
\quad 
p_S(N)=\sup\{p>1\;;\;\gamma_S(N,p)>0\}. 
\]
\begin{definition}
Let $f,g\in C_c^\infty(\R^N)$ and $p>1$. 
The function 
\[
u\in C([0,T);H^1(\R^N))\cap C^1([0,T);L^2(\R^N))
\cap L^p(0,T;L^p(\R^N))
\]
is called a weak solution 
of \eqref{eq:Str-demo} in $(0,T)$
if $u(0)=\ep f$, $\pa_t u(0)=\ep g$ and 
for every 
$\Psi\in C^\infty_c(\R^N\times [0,T))$, 
\begin{align*}
&\ep 
\int_{\R^N}g(x) \Psi(x,0)\,dx
+
\int_0^T\!\!\int_{\R^N}
|u(x,t)|^p\Psi(x,t)
\,dx\,dt
\\
&=
\int_0^T\!\!\int_{\R^N}
\Big(-\pa_tu(x,t) \pa_t\Psi(x,t)+\nabla u(x,t) \cdot\nabla \Psi(x,t)\Big)
\,dx\,dt.
\end{align*}
\end{definition}
\begin{proposition}
\label{prop:Str:sub}
Let $f,g$ satisfy \eqref{ass.initial} 
and let $u_\ep$ be a weak solution to \eqref{eq:Str-demo}
satisfying ${\rm supp}\,u_\ep\subset \{(x,t)\in \R^N\times [0,T]\;;\;|x|\leq r_0+t\}$
for $r_0=\sup\{|x|\;;\; x\in {\rm supp}(f,g)\}$.
Set $T_\ep$ as a lifespan of $u_\ep$ given by 
\[
T_\ep = \sup\{T>0\;;\;\text{there exists a solution to 
\eqref{eq:Str-demo} in $(0,T)$}\}. 
\]
If $1<p<p_S(N)$ (that is, $\gamma_S(N,p)\geq 0$), 
then $T_{\ep}<\infty$. Moreover, 
there exist $\ep_0>0$ and $C>0$
such that for every $\ep\in (0,\ep_0]$, 
\[
T_\ep
\leq 
\begin{cases}
C
\ep^{-\frac{p-1}{2}}
&
\text{if}\ N=1, \ 1<p<\infty,
\\[0pt]
C
\ep^{-\frac{p-1}{3-p}}
&
\text{if}\ N=2, \ 1<p\leq 2,
\\[0pt]
C\ep^{-2p(p-1)/\gamma_S(2,p)}
&
\text{if}\ N=2, \ 2<p<p_S(2),
\\[0pt]
C\ep^{-2p(p-1)/\gamma_S(N,p)}
&
\text{if}\ N\geq 3, \ 1<p<p_S(N).
\end{cases}
\]
\end{proposition}

\begin{proof}
If $T_\ep\leq 1$, then the assertion is trivial 
by choosing $\ep_0$ sufficiently small.  
Suppose that $T_\ep>1$ and take $T\in (1,T_\ep)$. 
Put 
$\eta\in C^\infty([0,\infty))$ satisfying 
\[
\eta(s)=
\begin{cases}
1 & s<1/2
\\
{\rm decreasing} & 1/2<s<1,
\\
0 & s>1,
\end{cases}
\quad 
\eta_T(s)=\eta(s/T).
\]
By the definition of weak solution 
of \eqref{eq:Str-demo} in $(0,T)$, 
we see from $\Psi=\eta_T(t)^{2p'}$ 
(multiplying compactly supported smooth function $\chi$
on $\R^N$
satisfying $\chi=1$ if $x \in B(0,r_0+T)$)
\begin{align*}
I[g]\ep
+
\int_0^T
\eta_T^{2p'}\int_{\R^N}
	|u_\ep|^p
\,dx
\,dt
&=
\int_0^T\!\!\int_{\R^N}
\Big(-\pa_tu_\ep \pa_t(\eta_T^{2p'})
+\nabla u_\ep \cdot\nabla (\eta_T^{2p'})\Big)
\,dx\,dt
\\
&=
\int_0^T\!\!\int_{\R^N}
u_\ep\pa_t^2(\eta_T^{2p'})
\,dx\,dt
\\
&=
\frac{2p'}{T^2}
\int_0^T
\eta_T^{2p'-2}
\int_{\R^N}
u_\ep\Big(\eta(t/T)\eta''(t/T)+(2p'-1)(\eta'(t/T))^2
\Big)
\,dx\,dt
\\
&
\leq 
\frac{2p'(\|\eta''\|_{L^\infty}+\|\eta'\|_{L^\infty}^2)}{T^2}
\int_0^T
\eta_T^{2p'/p}
\int_{\R^N}
	|u_\ep|
	\,dx
\,dt
\\
&
\leq 
\frac{[2p'(\|\eta''\|_{L^\infty}+\|\eta'\|_{L^\infty}^2)]^{p'}}{p'T^{2p'}}
\int_{0}^T\!\!
\int_{B(r_0+t)}
	\,dx
\,dt
+
\frac{1}{p}
\int_0^T
\eta_T^{2p'}
\int_{\R^N}
	|u_\ep|^p
	\,dx
\,dt,
\end{align*}
where we have used the finite propagation property.  
This yields  
\begin{align}
\label{eq:demo1}
p'I[g]\ep
+
\int_0^T
\eta_T^{2p'}
\int_{\R^N}
	|u_\ep|^p
\,dx
\,dt
&\leq 
C_1T^{N-1-\frac{2}{p-1}}, 
\quad 
C_1: =\frac{[2p'(\|\eta''\|_{L^\infty}+\|\eta'\|_{L^\infty}^2)]^{p'}(1+r_0)^{N+1}|S^{N-1}|}{N(N+1)}
\end{align}
with the volume of $N$-dimensional unit sphere $|S^{N-1}|$.
Since 
the choice of $T\in (1,T_\ep)$ is arbitrary, 
the above inequality implies 
the first and second estimates for $T_\ep$. 

To obtain the third and fourth estimates for $T_\ep$, 
we introduce 
a special solution to linear wave equation as follows: 
\[
w_{\lambda}(x,t)
=
\lambda^{N-1}\Big((\lambda+t)^2-|x|^2\Big)^{-\frac{N-1}{2}}, 
\quad 
\lambda >r_0. 
\]
Noting that $w_{\lambda}(x,0)\to 1$ 
and $\pa_t w_{\lambda}(x,0)\to 0$
as $\lambda\to\infty$
uniformly on ${\rm supp} (f,g)$, we see 
from dominated convergence theorem 
that 
there exists $\lambda_0>r_0$ such that 
\[
\int_{\R^N}
g(x)w_{\lambda_0}(x,0)
-
f(x)\pa_t w_{\lambda_0}(x,0)
\,dx
\geq 
\frac{1}{2}\int_{\R^N}
g(x)\,dx
=\frac{1}{2}I[g]>0. 
\]
Taking $\Psi=w_{\lambda_0}\eta_T^{2p'}$
(multiplying compactly supported smooth function $\chi$
on $\R^N$
satisfying $\chi=1$ if $x \in B(0,r_0+T)$) 
in the definition of weak solutions, we have
\begin{align*}
\ep \int_{\R^N}g(x) w_{\lambda_0}(x,0)\,dx
+
\int_0^T\!\!\int_{\R^N}
|u_\ep|^p\Psi
\,dx\,dt
&=
\int_0^T\!\!
\int_{\R^N}
\Big(-\pa_tu_\ep \pa_t\Psi+\nabla u_\ep\cdot\nabla\Psi\Big)
\,dx\,dt
\\
&=
\int_{\R^N}
f(x)\pa_tw_{\lambda_0}(x,0)
\,dx
+
\int_0^T\!\!
\int_{\R^N}
u_\ep
\Big(\pa_t^2\Psi-\Delta \Psi\Big)
\,dx\,dt.
\end{align*}
Neglecting the second term in the left-hand side 
and using H\"older's inequality 
and the definition of $w_{\lambda_0}$, 
we deduce 
\begin{align*}
I[g]\ep
&\leq 
2
\int_0^T\!\!
\int_{\R^N}
u_\ep
\Big(\pa_t^2\Psi-\Delta \Psi\Big)
\,dx\,dt
\\
&=
4p'
\int_0^T
\eta_T^{2p'-2}
\int_{\R^N}
u_\ep
\left(
2\pa_tw_{\lambda_0}\frac{\eta'(t/T)\eta(t/T)}{T}+w_{\lambda_0}\frac{\eta''(t/T)\eta(t/T)+(2p'-1)(\eta'(t/T))^2}{T^2}
\right)
\,dx\,dt
\\
&\leq 
C_2
T^{-N-1}
\int_{T/2}^T
\eta_T^{2p'/p}
\int_{\R^N}
|u_\ep|
\left(1-\frac{|x|^2}{(\lambda_0+t)^2}\right)^{-\frac{N+1}{2}}
\,dx\,dt
\\
&\leq 
C_2
T^{-N-1}
\left(
\int_{T/2}^T
\eta_T^{2p'}
\int_{\R^N}
|u_\ep|^p
\,dx\,dt
\right)^{\frac{1}{p}}
\left(
\int_{T/2}^T
\int_{B(0,r_0+t)}
\left(1-\frac{|x|}{\lambda_0+t}\right)^{-\frac{N+1}{2}p'}
\,dx\,dt
\right)^{\frac{1}{p'}}
\\
&\leq 
C_2'
\left(
T^{-N+\frac{N-1}{2}p}
\int_{0}^T
\eta_T^{2p'}
\int_{\R^N}
|u_\ep|^p
\,dx\,dt
\right)^{\frac{1}{p}}.
\end{align*}
for some $C_2>0$ and $C_2'>0$.
Therefore we have 
\begin{align}\label{eq:demo2}
\Big( I[g] \ep\Big)^p
T^{N-\frac{N-1}{2}p}
\leq 
C_2'\int_{0}^T
\eta_T^{2p'}
\int_{\R^N}
|u_\ep|^p
\,dx\,dt.
\end{align}
Combining \eqref{eq:demo1} and \eqref{eq:demo2}, we obtain 
\[
\Big( I[g]\ep\Big)^pT^{N-\frac{N-1}{2}p}
\leq 
C_1C_2'T^{N-1-\frac{2}{p-1}}
\]
which implies the third and fourth estimates for $T_\ep$. 
\end{proof}

\begin{remark}\label{rem:N=2,p=2}
The upper bound of $T_\ep$ is not sharp 
in the case $(N,p)=(2,2)$. 
Indeed, 
Lindblad \cite{Lindblad90}
Takamura \cite{Takamura15} proved the estimate
$T_\ep \leq C a(\ep)$ with $a^2\ep^2\log (1+a)=1$ 
by using a refined concentration estimate (similar to \eqref{eq:2.3}) 
which is deduced from pointwise estimates 
for solutions to linear wave equation. 
\end{remark}

\begin{remark}
In Yordanov--Zhang \cite{YZ06} and the subsequent papers, 
to prove a lower bound for $\int_{\R^N}|u|^p\,dx$, 
a positive radially symmetric solution $e^{-t}\phi(x)$
of $\pa_t^2u-\Delta u=0$
with the function $\phi$ satisfying $\phi-\Delta \phi=0$ 
were used. 
However, their treatment requires the positivity of initial value  
in the point-wise sense, in particular for high spatial dimensional cases. 
In contrast, 
the proof of Proposition \ref{prop:Str:sub} 
only needs the positivity of $I[g]$ by virtue of a new choice of solution $w_{\lambda}$. 
\end{remark}
\begin{remark}
In the proof of Proposition \ref{prop:Str:sub}, 
we do not use neither 
an auxiliary result 
for second order ordinary inequalities 
nor an iteration argument. 
The view-point from the proof of Proposition \ref{prop:Str:sub}
may give us an easier understanding 
about blowup phenomena for sub-critical case. 
\end{remark}

\section{Preliminaries for general cases}
\label{sec:Preliminaries}
To analyse more general equations and systems, 
we introduce 
the super-solutions to wave equations 
and self-similar solutions. 
In this section, we state the fundamental properties 
of the super-solutions to wave equations 
and self-similar solutions. 

We note that 
even if some notations overlap with ones in the previous 
section, 
we state them again for the reader's convenience.

\subsection{Super-solutions of wave equation and their properties}

First  we introduce 
super-solutions of wave equations. 

\begin{definition}\label{def:super-sol}
Let $(f,g)$ satisfy \eqref{ass.initial}. 
The function 
$u\in H^1(0,T;L^2(\R^N))\cap L^2(0,T;H^1(\R^N))$
is called a super-solution of $\pa_t^2u-\Delta u=H$ 
with $u(0)=\ep f$ and $\pa_t u(0)=\ep g$ 
and $H\in L^1(0,T;L^1(\R^N))$
if $u(0)=\ep f$ and 
\begin{align*}
\ep 
\int_{\R^N}
g(x)\Psi(x,0)\,dx
+
\int_0^T
\int_{\R^N}
	H\Psi
\,dx
\,dt
\leq 
\int_0^T
\int_{\R^N}
	(-\pa_t u \pa_t\Psi+\nabla u\cdot \nabla \Psi)
\,dx
\,dt
\end{align*}
for every nonnegative function $\Psi\in C^1_c(\R^N\times [0,T))$.
\end{definition}
Then we will use two kinds of families of 
cut-off functions with respect to time variables; 
$\eta\in C^\infty([0,\infty))$ satisfying 
\[
\eta(s)=
\begin{cases}
1 & s<1/2,
\\
{\rm decreasing} & 1/2<s<1,
\\
0 & s>1
\end{cases}
\]
and 
\[
\eta^*(s)=
\begin{cases}
0 & s<1/2,
\\
\eta(s) & s\geq 1/2
\end{cases}
\]
and for $k\geq 2$, $R>0$, 
\[
\eta_R(t)=\eta\left(\frac{t}{R}\right), 
\quad 
\eta_R^*(t)=\eta^*\left(\frac{t}{R}\right), 
\quad 
\psi_R(t)=[\eta_R(t)]^{k}, 
\quad 
\psi_R^*(t)=[\eta_R^*(t)]^{k}. 
\]
The functions $\eta_R^*$ 
and $\psi_R^*$ are used only to justify the following 
estimates.
\begin{lemma}\label{lem:cut}
Let $k\geq 2$ and $R\geq 1$. 
For every $t\geq 0$, 
\begin{align*}
|\pa_t\psi_R(t)|
&\leq 
\frac{k\|\eta'\|_{L^\infty}}{R}
[\psi_R^*(t)]^{1-\frac{1}{k}}, 
\quad
|\pa_t^2\psi_R(t)|
\leq 
\frac{k
\Big(
(k-1)\|\eta'\|_{L^\infty}^2
+\|\eta''\eta\|_{L^\infty}\Big)}{R^2}
[\psi_R^*(t)]^{1-\frac{2}{k}}.
\end{align*}
\end{lemma}
\begin{proof}
Noting that 
\begin{align*}
\pa_t\psi_R(t)
&=
k\eta_R'(t)[\eta_R(t)]^{k-1}
\\
\pa_t^2\psi_R(t)
&=
k\left((k-1)(\eta_R'(t))^2+\eta_R(t)\eta_R''(t)\right)[\eta_R(t)]^{k-2},
\end{align*}
we easily obtain the desired inequalities. 
\end{proof}
By using $\psi_R$ as a test function 
for super-solutions, we obtain 
the following lemma. 
\begin{lemma}\label{lem:key0}
Let $1<p<\infty$ and $k\geq 2p'$,  
and let $(f,g)$ satisfy \eqref{ass.initial} and 
let $u$ be a super-solution of $\pa_t^2u-\Delta u= H$ 
with $u(0)=\ep f$, $\pa_tu(0)=\ep g$, 
$H\in L^2(0,T;L^2(\R^N))$ 
and ${\rm supp}\,u\subset \{(x,t)\in \R^N\times [0,T]\;;\;|x|\leq r_0+t\}$
for $r_0=\sup\{|x|\;;\; x\in {\rm supp}(f,g)\}$.
Then the following inequalities hold:
\begin{itemize}
\item[\bf (i)] 
For every $1\leq R<T$
\begin{align}
I[g]\ep
+
\int_0^T
\int_{\R^N}
	H\psi_R
\,dx
\,dt&
\leq 
C_{1}R^{-2}\int_0^T\int_{\R^N}
	|u|[\psi_R^*]^{\frac{1}{p}}
\,dx
\,dt,
\end{align}
\item[\bf (ii)] 
For every $1\leq R<T$
\begin{align}
I[g]\ep
+
\int_0^T
\int_{\R^N}
	H\psi_R
\,dx
\,dt&
\leq 
C_{2}R^{-1}\int_0^T\int_{\R^N}
	|\pa_tu|[\psi_R^*]^{\frac{1}{p}}
\,dx
\,dt.
\end{align}
\end{itemize}
\end{lemma}
\begin{proof}
By the definition of super-solution of $\pa_t^2u-\Delta u=H\geq 0$, 
choosing $\Psi=\psi_R$ 
(with multiplying compactly supported smooth function $\zeta$
satisfying $\zeta\equiv 1$ on ${\supp\,u}$), we have 
\begin{align*}
I[g]\ep
+
\int_0^T
\int_{\R^N}
	H\psi_R
\,dx
\,dt
\leq 
-
\int_0^T
\int_{\R^N}
	\pa_t u \pa_t\psi_R
\,dx
\,dt.
\end{align*}
Then we can obtain {\bf (ii)} by using Lemma
\ref{lem:cut}
with $k\geq p'$. On the other hand, noting that 
\[
\int_{\R^N}
	\pa_t u \pa_t\psi_R
\,dx
=
\frac{d}{dt}
\int_{\R^N}
	u \pa_t\psi_R
\,dx
-
\int_{\R^N}
	u \pa_t^2\psi_R
\,dx,
\]
we see from \ref{lem:cut}
with $k\geq 2p'$. 
\end{proof}

For general (smooth) test function $\Psi$, 
we can find the following relation 
with respect to $\pa_t^2\Psi-\Delta\Psi$. 
\begin{lemma}\label{lem:test}
Let $u$ be a super-solution of $\pa_t^2u-\Delta u=H$ 
with $u(0)=\ep f$, $\pa_tu(0)=\ep g$ and $H\geq 0$. Then 
the following inequalities hold:
\begin{itemize}
\item[\bf (i)]
For $\Psi\in C^2(\R^N\times [0,T))$ satisfying $\Psi\geq 0$, 
\begin{align}\label{eq:2.1}
\ep
\int_{\R^N}
	(g\Psi(\cdot,0)-f\pa_t\Psi(\cdot,0))
\,dx+
\int_0^T\int_{\R^N}
	H\Psi
\,dx\,dt
\leq
\int_0^T\int_{\R^N}
	u(\pa_t^2\Psi-\Delta \Psi)
\,dx\,dt.
\end{align}
\item[\bf (ii)]
For $\widetilde{\Psi}\in C^3(\R^N\times [0,T))$ satisfying $\pa_t\widetilde{\Psi}\geq 0$, 
\begin{align}\label{eq:2.2}
\ep 
\int_{\R^N}
	(g\pa_t\widetilde{\Psi}(\cdot,0)-f\Delta \widetilde{\Psi}(\cdot,0))
\,dx
+
\int_0^T\int_{\R^N}
	H\pa_t\widetilde{\Psi}
\,dx\,dt
\leq 
\int_0^T\int_{\R^N}
	\pa_tu(\pa_t^2\widetilde{\Psi}-\Delta \widetilde{\Psi})
\,dx\,dt.
\end{align}
\end{itemize}
\end{lemma}
\begin{proof}
{\bf (i)}
Integration by parts yields that 
\begin{align*}
\int_0^T
\int_{\R^N}
	(-\pa_t u \pa_t\Psi+\nabla u\cdot \nabla \Psi)
\,dx
\,dt
&=
\frac{d}{dt}
\left[
	-\int_{\R^N}u\pa_t\Psi\,dx
\right]
+
\int_0^T
\int_{\R^N}
	(u\pa_t^2\Psi- \Delta u\Psi)
\,dx
\,dt
\\
&=
\ep 
\int_{\R^N}
	f\pa_t\Psi(x,0)
\,dx
+
\int_0^T
\int_{\R^N}
	u(\pa_t^2\Psi- \Delta \Psi)
\,dx
\,dt.
\end{align*}
Connecting the definition of super-solution, we deduce \eqref{eq:2.1}. 

\medskip 

{\bf (ii)} Applying {\bf (i)} $\Psi=\pa_t\widetilde{\Psi}\in C^2_c(\R^N\times [0,T))$, we have
\[
\ep 
\int_{\R^N}
	(g\pa_t\widetilde{\Psi}(\cdot,0)-f\pa_t^2\widetilde{\Psi}(\cdot,0))
\,dx+
\int_0^T\int_{\R^N}
	H\widetilde{\Psi}
\,dx\,dt
\leq
\int_0^T
\int_{\R^N}
	u\pa_t(\pa_t^2\widetilde{\Psi}-\Delta \widetilde{\Psi})
\,dx\,dt.
\]
Noting that 
\begin{align*}
\int_0^T
\int_{\R^N}
	u\pa_t(\pa_t^2\widetilde{\Psi}-\Delta \widetilde{\Psi})
\,dx\,dt
&=
\int_0^T
\frac{d}{dt}\left[
\int_{\R^N}
	u(\pa_t^2\widetilde{\Psi}-\Delta \widetilde{\Psi})
\,dx
\right]
\,dt
-
\int_0^T
\int_{\R^N}
	\pa_tu(\pa_t^2\widetilde{\Psi}-\Delta \widetilde{\Psi})
\,dx\,dt
\\
&=
-
\ep 
\int_{\R^N}
	f(\pa_t^2\widetilde{\Psi}(\cdot,0)-\Delta \widetilde{\Psi}(\cdot,0))
\,dx
\,dt
-
\int_0^T
\int_{\R^N}
	\pa_tu(\pa_t^2\widetilde{\Psi}-\Delta \widetilde{\Psi})
\,dx\,dt,
\end{align*}
we have \eqref{eq:2.2}. 
\end{proof}
The following two lemmas describe the concentration 
phenomena of the wave near the light cone $\pa B(0,1+t)$, 
which is essentially the same as an estimate 
given in Yordanov--Zhang \cite{YZ06}. 
In their proofs, we use a special solution 
of linear wave equation given by 
\[
V(x,t)=t
(t^2-|x|^2)^{-\frac{N+1}{2}}
=t^{-N}
\left(1-\frac{|x|^2}{t^2}\right)^{-\frac{N+1}{2}}
\]
in $\mathcal{Q}=\{(x,t)\in \R^N\times [0,\infty)\;;\;|x|< t\}$. 
By the notation in next subsection, we see $V(x,t)=\Phi_{\beta}(x,t)$ with $\beta=N$ (see \eqref{def:Phi} below). 

\begin{lemma}\label{lem:concentration}
Let $f,g$ satisfy 
\eqref{ass.initial}
and 
let $u$ be a super-solution of $\pa_t^2u-\Delta u=0$ 
with $u(0)=\ep f$, $\pa_tu(0)=\ep g$ 
and ${\rm supp}\,u\subset \{(x,t)\in \R^N\times [0,T]\;;\;|x|\leq r_0+t\}$
for $r_0=\sup\{|x|\;;\; x\in {\rm supp}(f,g)\}$. 
Then for every $p>1$ and $k\geq 2p'$, 
there exists a constant $\delta_1=\delta_1(N,p,k,f,g)>0$ 
(independent of $\ep$)
such that 
for every $1\leq R<T$
\begin{align}\label{eq:2.3}
\delta_1
\Big(
I[g]\ep
\Big)^{p}
R^{N-\frac{N-1}{2}p}
\leq 
\int_0^T\int_{\R^N}
	|u|^p\psi_R^*
\,dx
\,dt.
\end{align}
\end{lemma}

\begin{proof}
Put 
\[
v_\lambda(x,t)=\lambda^{N}V(x, \lambda+t),
\quad (x,t)\in \mathcal{Q}_\lambda=\{(x,t)\in\R^N\times [0,\infty)\;;\;(x,\lambda+t)\in \mathcal{Q}\}.
\] 
for $\lambda>r_0$ and then 
${\rm supp}\,u\subset \mathcal{Q}_\lambda$. 
Noting that 
\[
\pa_tV(x,t)
=
-t^{-N-1}\left(N+\frac{|x|^2}{t^2}\right)
\left(1-\frac{|x|^2}{t^2}\right)^{-\frac{N+3}{2}},
\] we see that 
\begin{align*}
&\int_{\R^N}
g(x)v_{\lambda}(x,0)
-f(x)\pa_tv_\lambda(x,0)
\,dx
\\
&=
\int_{\R^N}
g(x)\left(1-\frac{|x|^2}{\lambda^2}\right)^{-\frac{N+1}{2}}
\,dx
+
\frac{1}{\lambda}
\int_{\R^N}
f(x)
\left(N+\frac{|x|^2}{\lambda^2}\right)\left(1-\frac{|x|^2}{\lambda^2}\right)^{-\frac{N+3}{2}}
\,dx.
\end{align*}
Since the pair $(f,g)$ satisfies \eqref{ass.initial}, 
the dominated convergence theorem implies that
there exists $\lambda_0>r_0$ such that 
\begin{equation}\label{eq:positive}
\int_{\R^N}
g(x)v_{\lambda_0}(x,0)
-f(x)\pa_tv_{\lambda_0}(x,0)
\,dx
\geq 
\frac{1}{2}\int_{\R^N}
  g(x)
\,dx>0.
\end{equation}
On the other hand, 
since $u$ is a super-solution of $\pa_t^2u-\Delta u=0$, 
choosing $\Psi=v_{\lambda_0} \psi_R$ in Lemma \ref{lem:test} {\bf (i)}, 
we see from the fact 
$\pa_t^2v_{\lambda_0}-\Delta v_{\lambda_0}=0$ 
and 
Lemma \ref{lem:cut} that
\begin{align}
\nonumber
&
\ep\int_{\R^N}
g(x)v_{\lambda_0}(x,0)
-f(x)\pa_tv_{\lambda_0}(x,0)
\,dx
\\
\nonumber
&\leq 
\int_0^T
\int_{\R^N}
	u(\pa_t^2(v_{\lambda_0}\psi_R)-\Delta(v_{\lambda_0}\psi_R))
\,dx
\,dt
\\
\nonumber
&\leq 
C\int_0^T
\int_{\R^N}
	|u|
	\left(
		\frac{|\pa_tv_{\lambda_0}|}{R}+\frac{v_{\lambda_0}}{R^2}
	\right)
	[\psi_R^*]^{\frac{1}{p}}
\,dx
\,dt
\\
\label{eq:concent}
&\leq 
C
\left(
\int_0^T
\int_{\R^N}
	|u|^p\psi_R^*
\,dx
\,dt
\right)^{\frac{1}{p}}
\left(
\int_{\frac{R}{2}}^R
\int_{B(0,1+t)}
	\left(
		\frac{|\pa_tv_{\lambda_0}|}{R}+\frac{v_{\lambda_0}}{R^2}
	\right)^{p'}
\,dx
\,dt
\right)^{\frac{1}{p'}}.
\end{align}
Since $|x|\leq 1+t$ and $R/2\leq t\leq R$ yield
\[
\frac{|\pa_tv_{\lambda_0}|}{R}+\frac{v_{\lambda_0}}{R^2}
\leq 
CR^{-N-2}\left(1-\frac{|x|^2}{(\lambda_0+t)^2}\right)^{-\frac{N+3}{2}}
\leq 
CR^{-N-2}\left(1-\frac{|x|}{\lambda_0+t}\right)^{-\frac{N+3}{2}},
\]
a direct calculation implies 
\begin{align}
\label{eq:concent1}
\int_{\frac{R}{2}}^R
\int_{B(0,1+t)}
	\left(
		\frac{|\pa_tv_{\lambda_0}|}{R}+\frac{v_{\lambda_0}}{R^2}
	\right)^{p'}
\,dx
\,dt
&\leq 
CR^{N-(\frac{N+1}{2})p'}.
\end{align}
Combining \eqref{eq:positive}, 
\eqref{eq:concent}
and \eqref{eq:concent1}, 
we obtain \eqref{eq:2.3}.  
\end{proof}

\begin{lemma}\label{lem:concentration2}
Let $f,g$ satisfy 
\eqref{ass.initial}
and 
let $u$ be a super-solution of $\pa_t^2u-\Delta u=0$ 
with $u(0)=\ep f$, $\pa_tu(0)=\ep g$
and ${\rm supp}\,u\subset \{(x,t)\in \R^N\times [0,T]\;;\;|x|\leq r_0+t\}$
for $r_0=\sup\{|x|\;;\; x\in {\rm supp}(f,g)\}$. 
Then for every $q>1$ and $k\geq 2q'$, 
there exists a constant $\delta_1'=\delta_1'(N,q,k,f,g)>0$ 
such that 
for every $1\leq R<T$
\begin{align}\label{eq:2.6}
\delta_1'
\Big(
I[g]\ep
\Big)^{q}
R^{N-\frac{N-1}{2}q}
\leq 
\int_0^T\int_{\R^N}
	|\pa_tu|^q\psi_R^*
\,dx
\,dt
\end{align}
\end{lemma}
\begin{proof}
Set $w_{\lambda}(x,t)=-\lambda^{N+1} V_\lambda(x,t)=-\lambda v_\lambda(x,t)$ 
for $\lambda>1$
and then 
\[
\pa_t(w_{\lambda}\psi_R)
=
\lambda (-\pa_t v_{\lambda}\psi_R-v_{\lambda}\pa_t\psi_R)\geq 0
\]
Noting that $\pa_t^2w_{\lambda}-\Delta w_{\lambda}=0$, 
as in the proof of Lemma \ref{lem:concentration}, we can verify 
that there exists $\lambda'_0>1$ such that 
\begin{align*}
\int_{\R^N}
g(x)\pa_t w_{\lambda'_0}(x,0)
-f(x)\Delta w_{\lambda'_0}(x,0)
\,dx
&=
-
\lambda
\int_{\R^N}
g(x)\pa_t v_{\lambda'_0}(x,0)
\,dx
+
\lambda
\int_{\R^N}
f(x)\pa_t^2 v_{\lambda'_0}(x,0)
\,dx
\\
&\geq 
\frac{N}{2}I[g].
\end{align*}
Applying Lemma \ref{lem:test} {\bf (ii)} with $\widetilde{\Psi}=w_{\lambda'_0}\psi_R$, we have  
\begin{align}
\nonumber
\ep\int_{\R^N}
g(x) \pa_tv_{\lambda'_0}(x,0)
-f(x)\Delta v_{\lambda'_0}(x,0)
\,dx
&\leq 
-
\int_0^T
\int_{\R^N}
	\pa_tu(\pa_t^2(w_{\lambda_0'}\psi_R)-\Delta(w_{\lambda_0'}\psi_R))
\,dx
\,dt
\\
\nonumber
&\leq 
C\int_0^T
\int_{\R^N}
	|\pa_tu|
	\left(
		\frac{|\pa_tv_{\lambda_0'}|}{R}+\frac{v_{\lambda_0'}}{R^2}
	\right)
	[\psi_R^*]^{\frac{1}{p}}
\,dx
\,dt
\\
\label{eq:concent2}
&\leq 
C
R^{\frac{N-1}{2}-\frac{N}{q}}
\left(
\int_0^T
\int_{\R^N}
	|\pa_tu|^q\psi_R^*
\,dx
\,dt
\right)^{\frac{1}{q}}.
\end{align}
The above inequalities imply \eqref{eq:2.6}.
\end{proof}

\subsection{Self-similar solutions including Gauss's hyper geometric functions}

In the respective critical case of
blowup phenomena for wave equations, 
we need a precise information about 
the behavior of solutions to 
the linear wave equations. 
Therefore, next we introduce a family of 
self-similar solutions to $\pa_t^2u-\Delta u=0$ including Gauss's hypergeometric functions 
which also can be found in Zhou \cite{Zhou92, Zhou92-3,Zhou93}, Zhou--Han \cite{ZH14}
and also Ikeda--Sobajima \cite{IkedaSobajima1, IkedaSobajima2}. 
\begin{definition}\label{def:Phi}
Let $N\geq 2$. 
For $0< \beta<\infty$, define
\[
\Phi_{\beta}(x,t)
=
(t+|x|)^{-\beta}
F\left(\beta,\frac{N-1}{2},N-1;\frac{2|x|}{t+|x|}\right)
\quad 
(x,t)\in \mathcal{Q},
\]
where $F(a,b,c;z)$ is the Gauss hypergeometric function with a parameter $(a,b,c)$ given by 
\[
F(a,b,c;z)=\sum_{n=0}^\infty\frac{(a)_n(b)_n}{(c)_n}\,\frac{z^{n}}{n!}, \quad z\in [0,1)
\] 
with the The Pochhammer symbol $(d)_0=1$ and $(d)_n=\prod_{k=1}^n(d+k-1)$ for $n\in \N$. 
Also we set 
\[
\Phi_{\beta,\lambda}(x,t)
=
\lambda^{\beta}\Phi_{\beta}(x,\lambda+t), 
\quad (x,t)\in \mathcal{Q}_\lambda=\{(x,t)\in\R^N\times [0,\infty)\;;\;(x,\lambda+t)\in \mathcal{Q}\}.
\]
\end{definition}
\begin{remark}
In particular, the following formula for hypergeometric functions
is known: 
\[
F\left(a,a+\frac{1}{2},c;z\right)
=
(1+\sqrt{z})^{-2a}
F\left(2a,c-\frac{1}{2},2c-1;\frac{2\sqrt{z}}{1+\sqrt{z}}\right)
\]
(see Beals--Wong \cite[Section 8.9.6]{BW}). 
Putting $a=\frac{\beta}{2}$, $c=\frac{N}{2}$, we have
\[
\Phi_{\beta}(x,t)=(t+|x|)^{-\beta}
F\left(\beta,\frac{N-1}{2},N-1;\frac{2|x|}{t+|x|}\right)
=
t^{-\beta}F\left(\frac{\beta}{2},\frac{\beta+1}{2},\frac{N}{2};\frac{|x|^2}{t^2}\right). 
\]
This formula implies $\Phi_\beta\in C^\infty(\mathcal{Q})$. 
In one-dimensional case, 
the critical exponent for respective case
does not appear.
\end{remark}
Then the family $\{\Phi_\beta\}_{\beta>0}$ satisfies the following properties. 
For detail, see \cite{ZH14} and also \cite{IkedaSobajima1,IkedaSobajima2}.
\begin{lemma}\label{lem:Phi}
The following assertions hold:
\begin{itemize}
\item[\bf (i)] 
$\pa_t^2\Phi_\beta-\Delta \Phi_\beta=0$ on $\mathcal{Q}$.
\item[\bf (ii)] 
 $\pa_t\Phi_\beta=-\beta\Phi_{\beta+1}$ on $\mathcal{Q}$.
\item[\bf (iii)] 
 If $0\leq \beta< \frac{N-1}{2}$, then $t^{-\beta}\leq \Phi_\beta(x,t)\leq K_\beta t^{-\beta}$ on $\mathcal{Q}$.
\item[\bf (iv)] 
 If $\beta>\frac{N-1}{2}$, 
then 
\[
k_\beta t^{-\beta}\left(1-\frac{|x|^2}{t^2}\right)^{\frac{N-1}{2}-\beta}
\leq \Phi_\beta(x,t)\leq 
K_\beta t^{-\beta}\left(1-\frac{|x|^2}{t^2}\right)^{\frac{N-1}{2}-\beta}.
\]
\end{itemize}
\end{lemma}

In view of the properties of $\{\Phi_{\beta}\}_{\beta>0}$, 
we will take $\Psi=\Phi_{\beta,\lambda}\psi_R$. 
Then we have the following lemma.  

\begin{lemma}\label{lem:key}
Let $(f,g)$ satisfy \eqref{ass.initial} and 
let $u$ be a super-solution of $\pa_t^2u-\Delta u= H$ 
with $u(0)=\ep f$, $\pa_tu(0)=\ep g$, 
$H\in L^2(0,T;L^2(\R^N))$ and 
${\rm supp}\,u\subset \{(x,t)\in \R^N\times [0,T]\;;\;|x|\leq r_0+t\}$
for $r_0=\sup\{|x|\;;\; x\in {\rm supp}(f,g)\}$. 
Then for every $\beta>0$, 
there exists constants $\lambda_\beta>0$, $C_{\beta,1}>0$ and $C_{\beta,2}>0$  
such that the followings hold:
\begin{itemize}
\item[\bf (i)] 
If $k\geq 2p'$ and $\beta>0$, then  for every $\lambda\geq \lambda_\beta$ and $1\leq R<T$
\begin{align}
\frac{1}{2}I[g]\ep
+
\int_0^T
\int_{\R^N}
	H\Phi_{\beta,\lambda}\psi_R
\,dx
\,dt&
\leq 
C_{\beta,1}R^{-1}\int_0^T\int_{\R^N}
	|u|\Phi_{\beta+1,\lambda}[\psi_R^*]^{\frac{1}{p}}
\,dx
\,dt
\end{align}
\item[\bf (ii)]
If $k\geq 2p'$ and $\beta> 1$, then  for every $\lambda\geq \lambda_\beta$ and $1\leq R<T$
\begin{align}
\frac{1}{2}I[g]\ep
+
\int_0^T
\int_{\R^N}
	H\Phi_{\beta,\lambda}\psi_R
\,dx
\,dt&\leq 
C_{\beta,2}R^{-1}
\int_0^T
\int_{\R^N}
	|\pa_tu|\Phi_{\beta,\lambda}[\psi_R^*]^{\frac{1}{p}}
\,dx
\,dt
\end{align}
\end{itemize}
\end{lemma}
\begin{proof}
Put 
\[
c_{\beta,\lambda}(f,g)=
\int_{\R^N}
	g\Phi_{\beta,\lambda}+\beta f\Phi_{\beta+1,\lambda}
\,dx.
\]
Then we easily see that $c_{\beta,\lambda}(f,g)\to \int_{\R^N}g\,dx$
as $\lambda\to \infty$. 
Therefore there exists $\lambda_{\beta}>0$ such that 
for every $\lambda\geq \lambda_\beta$, 
\[
c_{\beta,\lambda}(f,g)\geq \frac{1}{2}I[g]. 
\]
Next, observe that for $R\geq 1$ and $(x,t)\in \mathcal{Q}_{\lambda}\cap {\rm supp}\,\psi_R^*$, 
\begin{align*}
 \Phi_{\beta,\lambda}(x,t)
&\leq 
2\lambda^{\beta}(\lambda+t)(\lambda+t+|x|)^{-\beta-1}
F\left(\beta,\frac{N-1}{2},N-1,\frac{2|x|}{\lambda+t+|x|}\right)
\\
&\leq 
2\lambda^{\beta}(\lambda+t)(\lambda+t+|x|)^{-\beta-1}
F\left(\beta+1,\frac{N-1}{2},N-1,\frac{2|x|}{\lambda+t+|x|}\right)
\\
&\leq 
2\lambda^{\beta+1}(1+t)(\lambda+t+|x|)^{-\beta-1}
F\left(\beta+1,\frac{N-1}{2},N-1,\frac{2|x|}{\lambda+t+|x|}\right)
\\
&\leq 
4R \Phi_{\beta+1,\lambda}(x,t). 
\end{align*}
Therefore by 
Lemma \ref{lem:cut} and 
Lemma \ref{lem:Phi} {\bf (i)}, {\bf (ii)}, 
\begin{align*}
|\pa_t^2(\Phi_{\beta,\lambda}\psi_R)-\Delta(\Phi_{\beta,\lambda}\psi_R)|
&\leq 
2|\pa_t\Phi_{\beta,\lambda}\pa_t\psi_R|
+\Phi_{\beta,\lambda}|\pa_t^2\psi_R|
\\
&\leq 
CR^{-1}\Phi_{\beta,\lambda+1}[\psi_R^*]^{1-\frac{2}{k}}.
\end{align*}
Choosing $\Psi=\Phi_{\beta,\lambda}\psi_R$ in 
Lemma \ref{lem:test}, we have {\bf (i)}. 
For {\bf (ii)}, 
we choose $\Psi=-(\beta-1)^{-1}\lambda\Phi_{\beta-1,\lambda}\psi_R$ for $\beta>1$ 
and $\lambda>\lambda_\beta$ (the same choice as the case {\bf (i)}). 
Noting that 
\[
\pa_t\Psi=\Phi_{\beta,\lambda}\psi_R-\frac{\lambda}{\beta-1}\Phi_{\beta,\lambda}\pa_t\psi_R
\geq \Phi_{\beta,\lambda}\psi_R, 
\]
we can deduce {\bf (ii)}. 
\end{proof}
Throughout the present paper 
we often use the following lemma. 
\begin{lemma}\label{lem:int-phi}
For every $R\geq 1$, 
\begin{align*}
\int_{\frac{R}{2}}^{R}
\int_{B(0,1+t)}
\Phi_{\beta,\lambda}^{p'}
\,dx
\,dt\leq 
\begin{cases}
C_{\lambda}R^{N+1-\beta p'}
&\text{if}\ \beta\in [0,\frac{N+1}{2}-\frac{1}{p}),
\\
C_{\lambda}R^{N-(\frac{N-1}{2})p'}\log R
&\text{if}\ \beta=\frac{N+1}{2}-\frac{1}{p},
\\
C_{\lambda}R^{N-(\frac{N-1}{2})p'}
&\text{if}\ \beta\in (\frac{N+1}{2}-\frac{1}{p},\infty).
\end{cases}
\end{align*}
\end{lemma}
\begin{proof}
All of the assertions are 
verified by using 
Lemma \ref{lem:Phi} {\bf (iii)} and {\bf (iv)}. 
\end{proof}
\subsection{Lemmas for lifespan estimates for respective critical cases}

To provide an upper bound of lifespan of solutions 
to respective critical case, we need to adopt 
the framework proposed in \cite{IkedaSobajima3}.
For the proof we refer 
the one of \cite[Proposition 2.1]{IkedaSobajima3}.

\begin{definition}
For nonnegative function $w\in L^1_{\rm loc}([0,T);L^1(\R^N))$, 
set\[
Y[w](R)=
\int_0^R
\left(
\int_0^T\int_{\R^N}
  w(x,t)\psi_{\sigma}^*(t)
\,dx\,dt
\right)\sigma^{-1}d\sigma, \quad R\in (0,T). 
\]
\end{definition}
Then $Y[w]$ has the following properties. 
\begin{lemma}\label{lem:int-to-diff}
For $w\in L^1_{\rm loc}([0,T);L^1(\R^N))$, 
$Y[w](\cdot)\in C^1((0,T))$ and for every $R\in (0,T)$
\begin{align*}
\frac{d}{dR}Y[w](R)
&=R^{-1}\int_0^T\int_{\R^N}
  w(x,t)\psi_R^*(t)
\,dx\,dt,
\\
Y[w](R)
&\leq 
\int_0^T\int_{\R^N}
  w(x,t)\psi_R(t)
\,dx\,dt. 
\end{align*}
\end{lemma}
It worth noticing that 
in critical cases 
the behavior of $Y[|u|^p\Phi_\beta]$ 
is crucial to obtain not only the blowup phenomena 
but also the upper bound of lifespan 
for respective problems. 
The following lemma provides the sharp 
upper bound of solutions 
for respective problems. 

\begin{lemma}\label{lem:lifespan-crit}
Let $2<t_0<T$, $0\leq \phi\in C^1([t_0,T))$.
Assume that 
\begin{align}\label{odi2}
\begin{cases}
\delta \leq K_1t\phi'(t),
&t\in (t_0,T),
\\
\phi(t)^{p_1} \leq K_2t(\log t)^{{p_2}-1}\phi'(t),
&t\in (t_0,T)
\end{cases}
\end{align}
with $\delta, K_1,K_2>0$ and and $p_1,p_2>1$.
If $p_2<p_1+1$, 
then there exists positive constants $\delta_0$ and $K_3$ (independent of $\delta$) such that 
\[
T\leq \exp(K_3\delta^{-\frac{p_1-1}{p_1-p_2+1}})
\]
when $0<\delta<\delta_0$.
\end{lemma}
\begin{proof}
If $T\leq t_0^4$, then 
we can choose 
$\delta_0'=\big(4K_0^{-1}\log t_0\big)^{-\frac{p_1-p_2+1}{p_1-1}}$. 
Therefore we assume $t_0^4<T$. 
By the first inequality in \eqref{odi2}, we have 
for every $t\in (t_0^2,T)$, 
\[
\phi(t)
=\phi(t^{1/2})+\int_{t^{1/2}}^{t}\phi'(s)\,ds
\geq \frac{\delta}{K_1}(\log t- \log t^{1/2})
=\frac{\delta}{2K_1}\log t.
\]
On the other hand, let $t_1\in (t_0^4,T)$ be arbitrary fixed.
The second inequality in \eqref{odi2} implies 
\begin{align*}
\frac{d}{dt}[\phi(t)]^{1-p_1}
\leq -\frac{p_1-1}{K_2}t^{-1}(\log t)^{1-p_2}, 
\quad 
t\in (t_0^2,T)
\end{align*}
and therefore integrating it over $[t_1^{1/2},t_1]$, we deduce 
\begin{align*}
[\phi(t_1)]^{1-p_1}
&\leq [\phi(t_1^{1/2})]^{1-p_1}
-\frac{p_1-1}{K_2}
\int_{t_1^{1/2}}^{t_1}s^{-1}(\log s)^{1-p_2}\,ds
\\
&\leq \left[\frac{\delta}{4K_1}\log t_1\right]^{1-p_1}
-\frac{p_1-1}{K_2}
\int_{1/2}^{1}\sigma^{1-p_2}\,d\sigma\ (\log t_1)^{2-p_2}
\\
&\leq 
\left(\left[\frac{\delta}{4K_1}\right]^{1-p_1}
-\frac{p_1-1}{K_2}
\int_{1/2}^{1}\sigma^{1-p_2}\,d\sigma
(\log t_1)^{1+p_1-p_2}\right)(\log t_1)^{1-p_1}.
\end{align*}
This yields that 
\[
(\log t_1)^{1+p_1-p_2}
\leq 
\frac{(4K_1)^{p_1-1}K_2}{p_1-1}
\left(\int_{1/2}^{1}\sigma^{1-p_2}\,d\sigma\right)^{-1}\delta^{-(p_1-1)}.
\]
Since the choice of $t_1\in (t_0^4,T)$ is arbitrary, we obtain the desired upper bound of $T$.
\end{proof}

\section{The case $\pa_t^2 u-\Delta u =G(u)$}\label{sec:Str}
The first problem is the following classical Cauchy problem 
of the following semilinear wave equation 
\begin{equation}\label{eq:Str}
\begin{cases}
\pa_t^2 u-\Delta u =G(u),
&
(x,t)\in \R^N\times (0,T),
\\
u(0)=\ep f,
&
x\in \R^N,
\\
\pa_tu(0)=\ep g,
&
x\in \R^N,
\end{cases}
\end{equation}
where the nonlinearity $G\in C^1(\R)$ satisfies
\[
G(0)=0, \quad G(s)\geq a|s|^{p}, \quad s\in \R
\]
for some $a>0$ and $p>1$. 
In this case the definition of weak solutions is the following: 
\begin{definition}
Let $f,g\in C_c^\infty(\R^N)$ and $p>1$. 
The function 
\[
u\in C([0,T);H^1(\R^N))\cap C^1([0,T);L^2(\R^N)), 
\quad 
G(u)\in L^1(0,T;L^1(\R^N))
\]
is called a weak solution 
of \eqref{eq:Str} in $(0,T)$
if $u(0)=\ep f$, $\pa_t u(0)=\ep g$ and 
for every 
$\Psi\in C^\infty_c(\R^N\times [0,T))$, 
\begin{align*}
&\ep 
\int_{\R^N}g(x) \Psi(x,0)\,dx
+
\int_0^T\!\!\int_{\R^N}
G(u(x,t))\Psi(x,t)
\,dx\,dt
\\
&=
\int_0^T\!\!\int_{\R^N}
\Big(-\pa_tu(x,t) \pa_t\Psi(x,t)+\nabla u(x,t) \cdot\nabla \Psi(x,t)\Big)
\,dx\,dt.
\end{align*}
\end{definition}
In order to give 
a unified view point with weakly coupled systems, we introduce 
\[
\Gamma_S(N,p)=\left(1+\frac{1}{p}\right)(p-1)^{-1}-\frac{N-1}{2}=\frac{\gamma_S(N,p)}{2p(p-1)}.
\]
The case of $G(s)=|s|^p$ with $1<p<p_S(N)$ 
was already shown in Section \ref{sec:Str-demo}. 
The essence of the proof for $1<p<p_S(N)$ 
is the same as in Section \ref{sec:Str-demo}.
Therefore we would state all assertions but  
prove only for the case $p=p_S(N)$. 
The assertion is formulated by the upper bound of 
maximal existence time of solutions to \eqref{eq:Str}.  
\begin{proposition}
\label{prop:Str}
Let $(f,g)$ satisfy \eqref{ass.initial} and 
let $u$ be a solution to \eqref{eq:Str} in $(0,T)$
satisfying ${\rm supp}\,u\subset \{(x,t)\in \R^N\times [0,T]\;;\;|x|\leq r_0+t\}$
for $r_0=\sup\{|x|\;;\; x\in {\rm supp}(f,g)\}$.
If 
\[
\Gamma_S(N,p)\geq 0
\]
(that is, $1<p\leq p_S(N)$), then $T$ has the following upper bound 
\[
T
\leq 
\begin{cases}
C
\ep^{-\frac{p-1}{2}}
&
\text{if}\ N=1, \ 1<p<\infty,
\\[2pt]
C
\ep^{-\frac{p-1}{3-p}}
&
\text{if}\ N=2, \ 1<p\leq 2,
\\[2pt]
C\ep^{-\Gamma_S(2,p)^{-1}}
&
\text{if}\ N=2, \ 2<p<p_S(2),
\\[2pt]
C\ep^{-\Gamma_S(N,p)^{-1}}
&
\text{if}\ N\geq 3, \ 1<p<p_S(N),
\\[2pt]
\exp(C\ep^{-p(p-1)})
&
\text{if}\ N\geq 2, \ p=p_S(N)
\end{cases}
\]
for every $\ep\in (0,\ep_0]$, where $\ep_0$ and $C$ are positive constants independent of $\ep$.
\end{proposition}
\begin{proof}
We only show the estimate for $T$ for the case $N\geq 2$, $p=p_S(N)$.
We will deduce differential inequalities 
for $Y=Y[|u|^p\Phi_{\beta,\lambda}]$ 
defined in Lemma \ref{lem:int-to-diff}
with $\beta=\beta_p=\frac{N-1}{2}-\frac{1}{p}$ 
and $\lambda=\lambda_{\beta_p}$ (given in Lemma \ref{lem:key}). 
By virtue of Lemma \ref{lem:int-to-diff}, 
we see from the inequality in Lemma \ref{lem:concentration} that 
\begin{align*}
Y'(R)R
&=
\int_0^T
\int_{\R^N}
|u|^p\Phi_{\beta,\lambda}\psi_R^*
\,dx
\,dt
\\
&\geq 
\left(\frac{\lambda}{\lambda+1}\right)^{\beta}
R^{-\beta}
\int_0^T
\int_{\R^N}
|u|^p\psi_R^*
\,dx
\,dt
\\
&
\geq 
\delta_1\left(\frac{\lambda}{\lambda+1}\right)^{\beta}
\Big(\ep
I[g]
\Big)^{p},
\end{align*}
where we used $\beta_p=N-\frac{N-1}{2}p$ 
by the assumption $\gamma_S(N,p)=0$. 
Moreover, 
using Lemma \ref{lem:key} {\bf (i)} with $\beta=\beta_p$
and Lemma \ref{lem:int-phi}, we have 
\begin{align*}
[Y(R)]^p
\leq 
\left(
\int_0^T\int_{\R^N}|u|^p\Phi_{\beta,\lambda}\psi_R
\,dx\,dt
\right)^{p}
&\leq 
C
R^{\frac{N-1}{2}p-N}(\log R)^{p-1}
\int_0^T
\int_{\R^N}
|u|^p\psi_R^*
\,dx
\,dt
\\
&\leq 
C(\log R)^{p-1}
\int_0^T
\int_{\R^N}
|u|^p\Phi_{\beta,\lambda}\psi_R^*
\,dx
\,dt
\\
&= 
C
R(\log R)^{p-1}Y'(R).
\end{align*}
Applying Lemma \ref{lem:lifespan-crit}
with $p_1=p_2=p$,
we obtain 
$T\leq \exp(C\ep^{-p(p-1)})$.
The proof is complete.
\end{proof}

\section{The case $\pa_t^2 u-\Delta u = G(\pa_tu)$}\label{sec:Gla}
In this section we consider 
the following semilinear wave equation with 
the nonlinearity governed by derivatives:
\begin{equation}\label{eq:Gla}
\begin{cases}
\pa_t^2 u-\Delta u = G(\pa_tu)
&
(x,t)\in \R^N\times (0,T) ,
\\
u(0)=\ep f
&
x\in \R^N,
\\
\pa_tu(0)=\ep g
&
x\in \R^N,
\end{cases}
\end{equation}
where the nonlinearity $G\in C^1(\R^N)$ satisfies 
\[
G(0)=0,\quad G(\sigma)\geq b|\sigma|^{p}, \quad \sigma\in\R
\]
for some $b>0$ and $p>1$. 
In this case the definition of weak solutions is the following: 
\begin{definition}
Let $f,g\in C_c^\infty(\R^N)$ and $p>1$. 
The function 
\[
u\in C([0,T);H^1(\R^N))\cap C^1([0,T);L^2(\R^N)), 
\quad 
G(\pa_t u)\in L^1(0,T;L^1(\R^N))
\]
is called a weak solution 
of \eqref{eq:Gla} in $(0,T)$
if $u(0)=\ep f$, $\pa_t u(0)=\ep g$ and 
for every 
$\Psi\in C^\infty_c(\R^N\times [0,T))$, 
\begin{align*}
&\ep 
\int_{\R^N}g(x) \Psi(x,0)\,dx
+
\int_0^T\!\!\int_{\R^N}
G(\pa_t u(x,t))\Psi(x,t)
\,dx\,dt
\\
&=
\int_0^T\!\!\int_{\R^N}
\Big(-\pa_tu(x,t) \pa_t\Psi(x,t)+\nabla u(x,t) \cdot\nabla \Psi(x,t)\Big)
\,dx\,dt.
\end{align*}
\end{definition}
For the problem \eqref{eq:Gla}, we set 
\[
\Gamma_G(N,p)=\frac{1}{p-1}-\frac{n-1}{2}.
\]
The exponent $p_G(N)=\frac{N+1}{N-1}$ ($\Gamma_G(N,p_G(N))=0$) 
is so-called Glassey exponent. 
For the convenience, we put $p_G(1)=\infty$. 
\begin{proposition}
\label{prop:Gla}
Let $(f,g)$ satisfy \eqref{ass.initial} and 
let $u$ be a solution to \eqref{eq:Gla} in $(0,T)$
satisfying ${\rm supp}\,u\subset \{(x,t)\in \R^N\times [0,T]\;;\;|x|\leq r_0+t\}$
for $r_0=\sup\{|x|\;;\; x\in {\rm supp}\,(f,g)\}$.
If 
\[\Gamma_G(N,p)\geq 0
\] 
then $T$ has the following upper bound 
\[
T
\leq 
\begin{cases}
C\ep^{-\Gamma_G(N,p)^{-1}}
&
{\rm if}\ 1<p<p_G(N),
\\[2pt]
\exp(C\ep^{-(p-1)})
&
{\rm if}\ N\geq 2, \ p=p_G(N)
\end{cases}
\]
for every 
$\ep\in (0,\ep_0]$, where $\ep_0$ and $C$ are positive constants independent of $\ep$.
\end{proposition}

\begin{proof}
Note that $u$ is a super-solution of $\pa_t^2u-\Delta u =b|\pa_tu|^p$. 
Choosing $\beta>\frac{N-1}{2}+1$ and $\lambda=\lambda_\beta$
in Lemma \ref{lem:key} {\bf (ii)} and Lemma \ref{lem:int-phi}, 
we have
\begin{align}
\nonumber
&
\frac{\ep}{2}I[g]
+
b\int_0^T
\int_{\R^N}
	|\pa_t u|^p\Phi_{\beta,\lambda}\psi_R
\,dx
\,dt
\\
\nonumber
&\leq 
CR^{-1}
\int_0^T
\int_{\R^N}
	|\pa_tu|\Phi_{\beta,\lambda}[\psi_R^*]^{\frac{1}{p}}
\,dx
\,dt
\\
\nonumber
&\leq 
CR^{-1}
\left(
\int_0^T
\int_{\R^N}
	|\pa_tu|^p\Phi_{\beta,\lambda}\psi_R^*
\,dx
\,dt
\right)^{\frac{1}{p}}
\left(
\int_{\frac{R}{2}}^R
\int_{B(0,r_0+t)}
	\Phi_{\beta,\lambda}
	\,dx
\,dt
\right)^{\frac{1}{p'}}
\\
\label{eq:4.mid}
&\leq 
CR^{-(\frac{1}{p-1}-\frac{N-1}{2})\frac{1}{p'}}
\left(
\int_0^T
\int_{\R^N}
	|\pa_tu|^p\Phi_{\beta,\lambda}\psi_R^*
\,dx
\,dt
\right)^{\frac{1}{p}}.
\end{align}
Setting $Y=Y[|\pa_tu|^p\Phi_{\beta,\lambda}]$ defined in Lemma \ref{lem:int-to-diff}, 
we obtain 
\begin{align*}
\left(
\frac{\ep}{2}I[g]
+bY(R)
\right)^p
\leq 
CR^{-\Gamma_G(N,p)(p-1)+1}
Y'(R). 
\end{align*}
Solving the above differential inequality, we can deduce
\[
R
\leq 
\begin{cases}
C\ep^{-(\frac{1}{p-1}-\frac{N-1}{2})^{-1}}
&
{\rm if}\ \Gamma_G(N,p)>0,
\\[5pt]
\exp(C\ep^{-(p-1)})
&
{\rm if}\ \Gamma_G(N,p)=0.
\end{cases}
\]
Since the choice of $R\in (1,T)$ is arbitrary, we could derive 
the desired upper bound of $T$.  
\end{proof}
\begin{remark}
We can also see from $\psi_R^*\leq\psi_R$ that 
if $\Gamma_G(N,p)>0$, then by Young's inequality we have
\[
\frac{\ep}{2}I[g]
\leq 
CR^{-(\frac{1}{p-1}-\frac{N-1}{2})}
\]
and therefore we can easily get 
the desired lifespan estimate for $1<p<\frac{N+1}{N-1}$. 
However, this argument does not work in the critical 
situation $p=\frac{N+1}{N-1}$.
\end{remark}
\section{The case of a combined type $\pa_t^2 u-\Delta u = G (u,\pa_tu)$}\label{sec:combined}

In this section we discuss 
the semilinear equation with 
nonlinearity of a combined type
\begin{equation}\label{eq:comb}
\begin{cases}
\pa_t^2 u-\Delta u = G (u,\pa_tu)
&(x,t)\in \R^N\times (0,T), 
\\
u(0)=\ep f
& x\in \R^N,
\\
\pa_tu(0)=\ep g
& x\in \R^N,
\end{cases}
\end{equation}
where the nonlinearity $G\in C^1(\R^2)$ satisfies
\[
G(0,0)=0, \quad 
G(s,\sigma)\geq a|s|^{q}+b|\sigma|^p
\quad (s,\sigma)\in \R^2
\]
for some $a,b>0$  and $p,q>1$. This problem 
has been considered by Zhou--Han \cite{ZH14-comb} 
Hidano--Wang--Yokoyama \cite{HWY16} and Wang--Zhou \cite{WZ18}.

In this case the definition of weak solutions is the following: 
\begin{definition}
Let $f,g\in C_c^\infty(\R^N)$ and $p>1$. 
The function 
\[
u\in C([0,T);H^1(\R^N))\cap C^1([0,T);L^2(\R^N)), 
\quad 
G(u, \pa_t u)\in L^1(0,T;L^1(\R^N))
\]
is called a weak solution 
of \eqref{eq:comb} in $(0,T)$
if $u(0)=\ep f$, $\pa_t u(0)=\ep g$ and 
for every 
$\Psi\in C^\infty_c(\R^N\times [0,T))$, 
\begin{align*}
&\ep 
\int_{\R^N}g(x) \Psi(x,0)\,dx
+
\int_0^T\!\!\int_{\R^N}
G\big(u(x,t),\pa_t u(x,t)\big)\Psi(x,t)
\,dx\,dt
\\
&=
\int_0^T\!\!\int_{\R^N}
\Big(-\pa_tu(x,t) \pa_t\Psi(x,t)+\nabla u(x,t) \cdot\nabla \Psi(x,t)\Big)
\,dx\,dt.
\end{align*}
\end{definition}
For the problem \eqref{eq:comb}, we set 
\[
\Gamma_{\rm comb}(N,p,q)
=
\frac{q+1}{p(q-1)}-\frac{N-1}{2}.
\]
The following assertion 
is already given by Hidano--Wang--Yokoyama \cite{HWY16}. 
\begin{proposition}
\label{prop:comb}
Let $(f,g)$ satisfy \eqref{ass.initial} and 
let $u$ be a solution to \eqref{eq:comb} in $(0,T)$
satisfying ${\rm supp}\,u\subset \{(x,t)\in \R^N\times [0,T)\;;\;|x|\leq r_0+t\}$
for $r_0=\sup\{|x|\;;\; x\in {\rm supp}(f,g)\}$.
If 
\[
\max\{\Gamma_S(N,q),\Gamma_G(N,p)\}\geq 0
\text{ or }
\Gamma_{\rm comb}(N,p,q)>0,
\] 
then $T$ has the following upper bound 
\[
T
\leq 
\begin{cases}
\exp(C\ep^{-(p-1)})
& 
\text {if}\ p=\frac{N+1}{N-1},\ q>1+\frac{4}{N-1},
\\[2pt]
C\ep^{-\Gamma_G(N,q)^{-1}}
& 
\text {if}\ p<\frac{N+1}{N-1},\ q>2p-1,
\\[2pt]
C\ep^{-\Gamma_{\rm comb}(N,p,q)^{-1}}
& 
\text {if}\ p\leq q\leq 2p-1,\ \Gamma_{\rm comb}(N,p,q)>0,
\\[2pt]
C\ep^{-\Gamma_S(N,p)^{-1}}
& 
\text {if}\ p>q, \quad q<p_S(N),
\\[2pt]
\exp(C\ep^{q(q-1)})
& 
\text {if}\ p\geq q=p_S(N)
\end{cases}
\]
for every 
$\ep\in (0,\ep_0]$, where $\ep_0$ and $C$ are positive constants independent of $\ep$.
\end{proposition}
\begin{remark}
In the case 
$\Gamma_S(N,q)<0, \Gamma_G(N,p)<0$ and 
$\Gamma_{\rm comb}(N,p,q)\leq 0$, 
Hidano--Wang--Yokoyama \cite{HWY16} 
proved global existence of 
small solutions to \eqref{eq:comb} when $N=2,3$. 
Therefore 
although it is open but one can expect that 
the same conclusion can be proved for all dimensions. 
\end{remark}
\begin{remark}
In Wang--Zhou \cite{WZ18}, 
the lower estimate for lifespan of solutions to \eqref{eq:comb} 
with $N=4$ and $p\in \{2\}\cup [3,\infty)$ 
is given. Therefore in these cases, the upper bound for $T$ in Proposition \ref{prop:comb}
is sharp.
\end{remark}
\begin{proof}
We have already proved the first, second, fourth and fifth cases 
in Propositions \ref{prop:Str} and \ref{prop:Gla}
because $G$ satisfies both $G\geq a|s|^q$
and $G\geq b|\sigma|^q$. Therefore we only consider 
the third case. 
Observe that $u$ is a super-solution of $\pa_t^2u-\Delta u=0$. 
By virtue of 
Lemma \ref{lem:concentration2}, we already have 
\begin{align*}
\delta_1'\Big(I[g]\ep \Big)^{p}R^{N-\frac{N-1}{2}p}
&\leq 
\int_0^T\int_{\R^N}
	|\pa_tu|^p\psi_R^*
\,dx\,dt.
\end{align*}
On the other hand, 
since $u$ is a super-solution of 
$\pa_t^2u-\Delta u\geq a|u|^q+b|\pa_tu|^p$, 
Lemma \ref{lem:key0} with Young's inequality implies 
\begin{align*}
I[g]\ep+
\int_0^T\int_{\R^N}
	(a|u|^q+b|\pa_tu|^p)\psi_R^*
\,dx\,dt
&\leq 
CR^{-2}
\int_0^T\int_{\R^N}
	|u|[\psi_R^*]^{\frac{1}{q}}
\,dx\,dt
\\
&\leq 
\frac{a^{-\frac{1}{q-1}}C^{q'}}{q'}R^{N-\frac{q+1}{q-1}}
+
\frac{a}{q}\int_0^T\int_{\R^N}
	|u|^q\psi_R^*
\,dx\,dt.
\end{align*}
Combining the above inequalities, we deduce
\begin{align*}
b\delta_1'\Big(I[g]\ep\Big)^{p}R^{N-\frac{N-1}{2}p}
\leq 
\frac{a^{-\frac{1}{q-1}}C^{q'}}{q'} R^{N-\frac{q+1}{q-1}}.
\end{align*}
Since the choice of $R\in (1,T)$ is arbitrary, 
this gives the third estimate for $T$. 
\end{proof}
\section{%
The case of the system 
$\pa_t^2 u-\Delta u =G_1(v)$ 
and 
$\pa_t^2 v-\Delta v =G_2(u)$}\label{sec:Str-Str}

The problem in this section is the following 
weakly coupled semilinear wave equations
\begin{equation}\label{sys:Str-Str}
\begin{cases}
\pa_t^2 u-\Delta u =G_1(v), 
&(x,t)\in \R^N\times(0,T), 
\\
\pa_t^2 v-\Delta v =G_2(u)
&(x,t)\in \R^N\times(0,T), 
\\
u(0)=\ep f_1
&x\in \R^N, 
\\
\pa_tu(0)=\ep g_1
&x\in \R^N, 
\\
v(0)=\ep f_2
&x\in \R^N, 
\\
\pa_tv(0)=\ep g_2
&x\in \R^N, 
\end{cases}
\end{equation}
where the nonlinearities $G_1\in C^1(\R)$ and $G_2\in C^1(\R)$ satisfy
\[
G_1(0)=0, \quad
G_2(0)=0, \quad
G_1(s)\geq a|s|^{p}, \quad 
G_2(s)\geq b|s|^{q}
\quad s\in\R
\]
for some $a,b>0$ and $p,q>1$. 
The problem \eqref{sys:Str-Str} 
with $G_1(s)=|s|^p$ and $G_2(s)=|s|^q$
is studied by 
Deng \cite{Deng99}, Kubo--Ohta \cite{KO99}, 
Agemi--Kurokawa--Takamura \cite{AKT00}, 
Kurokawa--Takamura--Wakasa \cite{KTW12}.
The aim of this section is to find the same result 
about upper bound of lifespan of solutions to \eqref{sys:Str-Str} 
by using a test function method 
similar to the one in Section \ref{sec:Str}.

In this case the definition of weak solutions is the following: 
\begin{definition}
Let $f_1,f_2, g_1,g_2\in C_c^\infty(\R^N)$. 
The pair of functions 
\begin{gather*}
(u,v)
\in C([0,T);(H^1(\R^N))^2)\cap C^1([0,T);(L^2(\R^N))^2), 
\\
G_2(u)\in L^1(0,T;L^1(\R^N)), 
\quad 
G_1(v)\in L^1(0,T;L^1(\R^N))
\end{gather*}
is called a weak solution 
of \eqref{sys:Str-Str} in $(0,T)$
if $(u,v)(0)=(\ep f_1,\ep f_2)$, 
$(\pa_t u, \pa_t v)(0)=(\ep g_1,\ep g_2)$ and 
for every 
$\Psi\in C^\infty_c(\R^N\times [0,T))$, 
\begin{align*}
&\ep 
\int_{\R^N}g_1(x) \Psi(x,0)\,dx
+
\int_0^T\!\!\int_{\R^N}
G_1\big(v(x,t)\big)\Psi(x,t)
\,dx\,dt
\\
&=
\int_0^T\!\!\int_{\R^N}
\Big(-\pa_tu(x,t) \pa_t\Psi(x,t)+\nabla u(x,t) \cdot\nabla \Psi(x,t)\Big)
\,dx\,dt,
\\
&\ep 
\int_{\R^N}g_2(x) \Psi(x,0)\,dx
+
\int_0^T\!\!\int_{\R^N}
G_2\big(u(x,t)\big)\Psi(x,t)
\,dx\,dt
\\
&=
\int_0^T\!\!\int_{\R^N}
\Big(-\pa_tu(x,t) \pa_t\Psi(x,t)+\nabla u(x,t) \cdot\nabla \Psi(x,t)\Big)
\,dx\,dt. 
\end{align*}
\end{definition}
As in the previous works listed above, we introduce 
\[
F_{SS}(N,p,q)=\left(p+2+\frac{1}{q}\right)(pq-1)^{-1}-\frac{N-1}{2}.
\]
The assertion for the estimates for $T$ is the following. 
The result has been given until 
Kurokawa--Takamura--Wakasa \cite{KTW12}. 
\begin{proposition}\label{prop:Str-Str}
Let $(f_1,g_1)$ and $(f_2,g_2)$ satisfy \eqref{ass.initial} and 
let $(u,v)$ be a weak solution of the system \eqref{sys:Str-Str}
satisfying ${\rm supp}(u,v)\subset \{(x,t)\in \R^N\times [0,T)\;;\;|x|\leq r_0+t\}$
for $r_0=\sup\{|x|\;;\; x\in {\rm supp}\,(f_1,f_2,g_1,g_2)\}$.
If 
\[
\Gamma_{SS}(N,p,q)=\max\{F_{SS}(N,p,q),F_{SS}(N,q,p)\}\geq 0,
\] 
then $T$ has the following upper bound 
\[
T
\leq 
\begin{cases}
C\ep^{-\Gamma_{SS}(N,p,q)^{-1}}
&
\text{if }\ \Gamma_{SS}(N,p,q)>0,
\\
\exp(C\ep^{-\min\{p(pq-1),q(pq-1)\}})
&
\text{if }\ \Gamma_{SS}(N,p,q)=0,\ p\neq q,
\\
\exp(C\ep^{-p(p-1)})
&
\text{if }\ \Gamma_{SS}(N,p,q)=0,\ p= q
\end{cases}
\]
for every 
$\ep\in (0,\ep_0]$, where $\ep_0$ and $C$ are positive constants independent of $\ep$.
\end{proposition}
\begin{proof}
We assume $F_{SS}(N,p,q)\geq F_{SS}(N,q,p)$, otherwise, we can interchange $u$ and $v$.
Moreover, we already have the following estimates by Lemma \ref{lem:concentration}:
\begin{align}
\label{eq:SS1}
\delta_1
\Big(I[g_1]\ep\Big)^{q}
R^{N-\frac{N-1}{2}q}
&\leq 
\int_0^T\int_{\R^N}
	|u|^q\psi_R^*
\,dx
\,dt.
\end{align}
Now we consider the case $F_{SS}(N,p,q)>0$. 
By Lemma \ref{lem:key0},
we have
\begin{align}
\nonumber
\left(\int_0^T\!\!\int_{\R^N}|v|^p\psi_R
\,dx\,dt\right)^{q}
&\leq 
CR^{-2+(N-1)(q-1)}
\int_0^T\!\!\int_{\R^N}|u|^q\psi_R^*\,dx\,dt,
\\
\label{eq:SS2}
\left(\int_0^T\!\!\int_{\R^N}|u|^q\psi_R
\,dx\,dt\right)^{p}
&\leq 
CR^{-2+(N-1)(p-1)}
\int_0^T\!\!\int_{\R^N}|v|^p\psi_R^*\,dx\,dt. 
\end{align}
These imply 
\begin{align*}
\left(
\int_0^T\!\!\int_{\R^N}|u|^q\psi_R
\,dx\,dt
\right)^{pq}
&\leq
C\left(R^{-2+(N-1)(p-1)}
\int_{\R^N}|v|^p\psi_R^*\,dx\,dt\right)^{q}
\\
&\leq 
CR^{[-2+(N-1)(p-1)]q-2+(N-1)(q-1)}
\int_0^T\!\!\int_{\R^N}|u|^q\psi_R^*\,dx\,dt
\\
&\leq 
CR^{(N-1)(pq-1)-2(q+1)}
\int_0^T\!\!\int_{\R^N}|u|^q\psi_R^*\,dx\,dt,
\end{align*}
and hence
\[
\int_0^T\!\!\int_{\R^N}|u|^q\psi_R
\,dx\,dt
\leq
CR^{N-1-\frac{2(q+1)}{pq-1}}
=
CR^{N-\frac{pq+2q+1}{pq-1}}. 
\]
Combining \eqref{eq:SS1}, we deduce
\begin{align*}
\Big(I[g_1]\ep\Big)^{q}
\leq 
CR^{\frac{N-1}{2}q-\frac{pq+2q+1}{pq-1}}
=
CR^{-qF_{SS}(N,p,q)}. 
\end{align*}
Since $R\in (1,T)$ is arbitrary, we have the upper bound for $T$. 

Next we consider the critical case $F_{SS}(N,p,q)=0$. 
If $F_{SS}(N,p,q)=F_{SS}(N,q,p)$, then we have $p=q=p_S(N)$. 
In this case, we consider the following 
differential inequality
\[
\pa_t^2(u+v)-\Delta (u+v)=b|u|^p+a|v|^{p}
\geq 2^{-p}\min\{a,b\}(|u|+|v|)^p
\geq 2^{-p}\min\{a,b\}|u+v|^p.
\]
Applying Proposition \ref{prop:Str} with $a$ 
replaced with $2^{-p}\min\{a,b\}$, 
we can obtain $T\leq \exp(C\ep^{-p(p-1)})$. 
Here we assume $0=F_{SS}(N,p,q)>F_{SS}(N,q,p)$. 
Then combining \eqref{eq:SS1} and \eqref{eq:SS2}, we have
\[
\int_0^T\!\!\int_{\R^N}|v|^p\psi_R^*\,dx\,dt
\geq
\delta_1'\Big(I[g_1]\ep\Big)^{pq} R^{(N-\frac{N-1}{2}q)p+2-(N-1)(p-1)}
=
\delta_1'\Big(I[g_1]\ep\Big)^{pq} R^{\frac{N-1}{2}-\frac{1}{q}}
\]
where we have used $F_{SS}(N,p,q)=0$. We see from Lemma \ref{lem:Phi} {\bf (iii)} that 
\[
\int_0^T\!\!\int_{\R^N}|v|^p\Phi_{\beta,\lambda}\psi_R^*\,dx\,dt
\geq
\delta_1'\Big(I[g_1]\ep\Big)^{pq}
\] 
with $\beta=\beta_q=\frac{N-1}{2}-\frac{1}{q}$ and $\lambda=\lambda_{\beta_q}$. 
By using Lemma \ref{lem:key} {\bf (i)} with $\beta=\beta_q$, 
Lemma \ref{lem:int-phi}, \eqref{eq:SS2} 
and the condition $F_{SS}(N,p,q)=0$, we have
\begin{align*}
\left(
\int_0^T\!\!\int_{\R^N}|v|^p\Phi_{\beta,\lambda}\psi_R\,dx\,dt
\right)^{pq}
&
\leq 
CR^{-(N-\frac{N-1}{2}q)p}(\log R)^{p(q-1)}
\left(
\int_0^T\int_{\R^N}|u|^q\psi_R^*\,dx\,dt
\right)^{p}
\\
&\leq 
CR^{-(N-\frac{N-1}{2}q)p-2+(N-1)(p-1)}(\log R)^{p(q-1)}
\int_0^T\!\!\int_{\R^N}|v|^p\psi_R^*\,dx\,dt
\\
&\leq 
C(\log R)^{p(q-1)}\int_0^T\!\!\int_{\R^N}|v|^p\Phi_{\beta,\lambda}\psi_R^*\,dx\,dt.
\end{align*}
In view of Lemma \ref{lem:int-to-diff}, taking $Y=Y[|v|^p\Phi_{\beta,\lambda}]$, 
we deduce 
\[	
\begin{cases}
\ep^{pq}\leq CRY'(R), 
\\
[Y(R)]^{pq}\leq CR(\log R)^{p(q-1)}Y'(R).
\end{cases}
\]
Applying Lemma \ref{lem:lifespan-crit} with 
$\delta=\ep^{pq}$, $p_1=pq$ and $p_2=p(q-1)+1$, we obtain 
\[
T \leq \exp(C\ep^{-q(pq-1)}).
\]
The proof is complete. 
\end{proof}

\section{The case of the system $\pa_t^2 u-\Delta u = G_1(\pa_tv)$ and $\pa_t^2 v-\Delta v =G_2(\pa_tu)$}\label{sec:Gla-Gla}
We consider the 
following weakly coupled system 
of semilinear wave equations 
with nonlinearities including derivatives
\begin{equation}\label{sys:Gla-Gla}
\begin{cases}
\pa_t^2 u-\Delta u =G_1(\pa_tv), 
&(x,t)\in \R^N\times(0,T), 
\\
\pa_t^2 v-\Delta v =G_2(\pa_tu)
&(x,t)\in \R^N\times(0,T), 
\\
u(0)=\ep f_1
&x\in \R^N, 
\\
\pa_tu(0)=\ep g_1
&x\in \R^N, 
\\
v(0)=\ep f_2
&x\in \R^N, 
\\
\pa_tv(0)=\ep g_2
&x\in \R^N, 
\end{cases}
\end{equation}
where the nonlinearities $G_1\in C^1(\R)$ and $G_2\in C^1(\R)$ satisfy
\[
G_1(0)=0, \quad
G_2(0)=0, \quad
G_1(\sigma)\geq a|\sigma|^{p}, \quad 
G_2(\sigma)\geq b|\sigma|^{q}, 
\quad \sigma\in\R
\]
for some $a,b>0$  and $p,q>1$. 
The blowup phenomena of the system \eqref{sys:Gla-Gla} 
is studied in Deng \cite{Deng99}. 
It seems that 
the upper bound of lifespan of solutions to \eqref{sys:Gla-Gla} 
has not been obtained so far. 
In the present paper 
we obtain an upper bound of lifespan 
by our technique similar to Section \ref{sec:Gla}. 

In this case the definition of weak solutions is the following: 
\begin{definition}
Let $f_1,f_2, g_1,g_2\in C_c^\infty(\R^N)$. 
The pair of functions 
\begin{gather*}
(u,v)
\in C([0,T);(H^1(\R^N))^2)\cap C^1([0,T);(L^2(\R^N))^2), 
\\
G_2(\pa_tu)\in L^1(0,T;L^1(\R^N)), 
\quad 
G_1(\pa_tv)\in L^1(0,T;L^1(\R^N))
\end{gather*}
is called a weak solution 
of \eqref{sys:Gla-Gla} in $(0,T)$
if $(u,v)(0)=(\ep f_1,\ep f_2)$, 
$(\pa_t u, \pa_t v)(0)=(\ep g_1,\ep g_2)$ and 
for every 
$\Psi\in C^\infty_c(\R^N\times [0,T))$, 
\begin{align*}
&\ep 
\int_{\R^N}g_1(x) \Psi(x,0)\,dx
+
\int_0^T\!\!\int_{\R^N}
G_1\big(\pa_tv(x,t)\big)\Psi(x,t)
\,dx\,dt
\\
&=
\int_0^T\!\!\int_{\R^N}
\Big(-\pa_tu(x,t) \pa_t\Psi(x,t)+\nabla u(x,t) \cdot\nabla \Psi(x,t)\Big)
\,dx\,dt,
\\
&\ep 
\int_{\R^N}g_2(x) \Psi(x,0)\,dx
+
\int_0^T\!\!\int_{\R^N}
G_2\big(\pa_tu(x,t)\big)\Psi(x,t)
\,dx\,dt
\\
&=
\int_0^T\!\!\int_{\R^N}
\Big(-\pa_tu(x,t) \pa_t\Psi(x,t)+\nabla u(x,t) \cdot\nabla \Psi(x,t)\Big)
\,dx\,dt. 
\end{align*}
\end{definition}
In this case, set 
\[
F_{GG}(N,p,q)
=\frac{p+1}{pq-1}-\frac{N-1}{2}.
\]
The exponent $F_{GG}$ 
seems to play the same rule (with the shift of dimension $N$ to $N-1$)
as the one for weakly coupled heat equations in Escobedo--Herrero \cite{EH91}
(see also Nishihara--Wakasugi \cite{NW14} for weakly coupled damped wave equations). 
\begin{proposition}\label{prop:Gla-Gla}
Let $(f_1,g_1)$ and $(f_2,g_2)$ satisfy \eqref{ass.initial} and 
let $(u,v)$ be a weak solution of the system \eqref{sys:Gla-Gla}
satisfying ${\rm supp}(u,v)\subset \{(x,t)\in \R^N\times [0,T)\;;\;|x|\leq r_0+t\}$
for $r_0=\sup\{|x|\;;\; x\in {\rm supp}\,(f_1,f_2,g_1,g_2)\}$.
If 
\[
\Gamma_{GG}(N,p,q)=\max\{F_{GG}(N,p,q),F_{GG}(N,q,p)\}\geq 0,
\] 
then $T$ has the following upper bound 
\[
T\leq 
\begin{cases}
C\ep^{-\Gamma_{GG}(N,p,q)^{-1}}
&
{\rm if }\ \Gamma_{GG}(N,p,q)> 0,
\\
\exp(C\ep^{-(pq-1)})
&
{\rm if }\ \Gamma_{GG}(N,p,q)= 0, \quad p\neq q,
\\
\exp(C\ep^{-(p-1)})
&
{\rm if }\ \Gamma_{GG}(N,p,q)= 0, \quad p= q 
\end{cases}
\]
for every $\ep\in (0,\ep_0]$, where $\ep_0$ and $C$ are positive constants independent of $\ep$.
\end{proposition}

\begin{proof}
As in the proof of Proposition \ref{prop:Str-Str}, 
we only consider the case $F_{GG}(N,p,q)\geq F_{GG}(N,q,p)$ 
(that is, $p\geq q$). 
Lemma \ref{lem:key} {\bf (ii)} with $\beta>\frac{N-1}{2}+1$
and $\lambda=\lambda_{\beta}$ 
and Lemma \ref{lem:int-phi} imply
\begin{align*}
\frac{\ep}{2}I[g_1]
+
a\int_0^T\!\!
\int_{\R^N}
	|\pa_t v|^p\Phi_{\beta,\lambda}\psi_R
\,dx
\,dt
&\leq 
CR^{-1}
\int_0^T\!\!
\int_{\R^N}
	|\pa_tu|\Phi_{\beta,\lambda}[\psi_R^*]^{\frac{1}{q}}
\,dx
\,dt
\\
&\leq 
CR^{-(\frac{1}{q-1}-\frac{N-1}{2})\frac{1}{q'}}
\left(
\int_0^T\!\!
\int_{\R^N}
	|\pa_tu|^q\Phi_{\beta,\lambda}\psi_R^*
\,dx
\,dt
\right)^{\frac{1}{q}},
\end{align*}
and similarly, 
\begin{align*}
\frac{\ep}{2}I[g_2]
+
b\int_0^T\!\!
\int_{\R^N}
	|\pa_t u|^q\Phi_{\beta,\lambda}\psi_R
\,dx
\,dt
&\leq 
CR^{-1}
\int_0^T\!\!
\int_{\R^N}
	|\pa_tv|\Phi_{\beta,\lambda}[\psi_R^*]^{\frac{1}{p}}
\,dx
\,dt
\\
&\leq 
CR^{-(\frac{1}{p-1}-\frac{N-1}{2})\frac{1}{p'}}
\left(
\int_0^T\!\!
\int_{\R^N}
	|\pa_tv|^p\Phi_{\beta,\lambda}\psi_R^*
\,dx
\,dt
\right)^{\frac{1}{p}}.
\end{align*}
Combining these inequalities, we deduce
\begin{align*}
\left(
\frac{\ep}{2}I[g_1]
+
a\int_0^T\!\!
\int_{\R^N}
	|\pa_t v|^p\Phi_{\beta,\lambda}\psi_R
\,dx
\,dt
\right)^{pq}
&\leq 
CR^{(\frac{N-1}{2})(pq-1)-p-1}
\int_0^T\!\!
\int_{\R^N}
	|\pa_t v|^p\Phi_{\beta,\lambda}\psi_R^*
\,dx
\,dt. 
\end{align*}
Using $Y=Y[|\pa_t v|^p\Phi_{\beta,\lambda}]$ in Lemma \ref{lem:int-to-diff}, 
we can verify 
\[
T\leq 
\begin{cases}
C\ep^{-F_{GG}(N,p,q)^{-1}}
&
{\rm if }\ F_{GG}(N,p,q)> 0,
\\
\exp(C\ep^{-(pq-1)})
&
{\rm if }\ F_{GG}(N,p,q)= 0.
\end{cases}
\]
Note that in the case 
$F_{GG}(N,p,q)=F_{GG}(N,q,p)=0$, we have $p=q$ and then $\Gamma_G(N,p)=0$. 
Applying Proposition \ref{prop:Gla} to the inequality 
\[
\pa_t^2(u+v)-\Delta (u+v)\geq 2^{-p}\min\{a,b\}|\pa_t(u+v)|^{p},
\] we have $T \leq \exp(C\ep^{-(p-1)})$. 
\end{proof}

\section{The case of system $\pa_t^2 u-\Delta u = G_1(v)$ and $\pa_t^2 v-\Delta v = G_2(\pa_tu)$}\label{sec:Str-Gla}
To close the paper, in the last section
we consider the weakly coupled 
system of 
semilinear wave equations
of the form 
\begin{equation}\label{sys:Str-Gla}
\begin{cases}
\pa_t^2 u-\Delta u =G_1(v), 
&(x,t)\in \R^N\times(0,T), 
\\
\pa_t^2 v-\Delta v =G_2(\pa_tu)
&(x,t)\in \R^N\times(0,T), 
\\
u(0)=\ep f_1
&x\in \R^N, 
\\
\pa_tu(0)=\ep g_1
&x\in \R^N, 
\\
v(0)=\ep f_2
&x\in \R^N, 
\\
\pa_tv(0)=\ep g_2
&x\in \R^N, 
\end{cases}
\end{equation}
where the nonlinearities $G_1\in C^1(\R)$ and $G_2\in C^1(\R)$ satisfy
\[
G_1(0)=0, 
\quad
G_2(0)=0,
\quad
G_1(s)\geq a|s|^{q}, \quad 
G_2(\sigma)\geq b|\sigma|^{p}, 
\quad s,\sigma\in \R
\]
for some $a,b>0$ and $p,q>1$. 
In this case the definition of weak solutions is the following: 
\begin{definition}
Let $f_1,f_2, g_1,g_2\in C_c^\infty(\R^N)$. 
The pair of functions $(u,v)$ 
\begin{gather*}
(u,v)
\in C([0,T);(H^1(\R^N))^2)\cap C^1([0,T);(L^2(\R^N))^2), 
\\
G_2(\pa_tu)\in L^1(0,T;L^1(\R^N)), 
\quad 
G_1(v)\in L^1(0,T;L^1(\R^N))
\end{gather*}
is called a weak solution 
of \eqref{sys:Str-Gla} in $(0,T)$
if $(u,v)(0)=(\ep f_1,\ep f_2)$, 
$(\pa_t u, \pa_t v)(0)=(\ep g_1,\ep g_2)$ and 
for every 
$\Psi\in C^\infty_c(\R^N\times [0,T))$, 
\begin{align*}
&\ep 
\int_{\R^N}g_1(x) \Psi(x,0)\,dx
+
\int_0^T\!\!\int_{\R^N}
G_1\big(v(x,t)\big)\Psi(x,t)
\,dx\,dt
\\
&=
\int_0^T\!\!\int_{\R^N}
\Big(-\pa_tu(x,t) \pa_t\Psi(x,t)+\nabla u(x,t) \cdot\nabla \Psi(x,t)\Big)
\,dx\,dt,
\\
&\ep 
\int_{\R^N}g_2(x) \Psi(x,0)\,dx
+
\int_0^T\!\!\int_{\R^N}
G_2\big(\pa_t u(x,t)\big)\Psi(x,t)
\,dx\,dt
\\
&=
\int_0^T\!\!\int_{\R^N}
\Big(-\pa_tu(x,t) \pa_t\Psi(x,t)+\nabla u(x,t) \cdot\nabla \Psi(x,t)\Big)
\,dx\,dt. 
\end{align*}
\end{definition}
Here we introduce 
two kind of exponent for the problem \eqref{sys:Str-Gla}. 
\begin{align*}
F_{SG,1}(N,p,q)&=\left(\frac{1}{p}+1+q\right)(pq-1)^{-1}-\frac{N-1}{2}, 
\\
F_{SG,2}(N,p,q)&=\left(2+\frac{1}{q}\right)(pq-1)^{-1}-\frac{N-1}{2}.
\end{align*}
The problem \eqref{sys:Str-Gla} is recently discussed in Hidano--Yokoyama \cite{HY16} 
and the blowup phenomena for small solutions are shown in the case $F_{SG,1}(N,p,q)>0$. 
The other condition $F_{SG,2}(N,p,q)\geq 0$ is now carried out by our test function method. 
Furthermore, we can also find the lifespan estimate 
for $(p,q)$ on the borderline case. 
\begin{proposition}\label{prop:Str-Gla}
Let $(f_1,g_1)$ and $(f_2,g_2)$ satisfy \eqref{ass.initial} and 
let $(u,v)$ be a weak solution of the system \eqref{sys:Str-Gla}
satisfying ${\rm supp}(u,v)\subset \{(x,t)\in \R^N\times [0,T)\;;\;|x|\leq r_0+t\}$
for $r_0=\sup\{|x|\;;\; x\in {\rm supp}\,(f_1,f_2,g_1,g_2)\}$.
If 
\[
\Gamma_{SG}(N,p,q)=\max\{F_{SG,1}(N,p,q),F_{SG,2}(N,p,q)\}\geq 0, 
\]
then $T$ has the following upper bound: 
\[
T\leq 
\begin{cases}
C\ep^{-\Gamma_{SG}(N,p,q)^{-1}} 
&
\text{if}\ 
\Gamma_{SG}(N,p,q)>0,
\\
\exp(C\ep^{-q(pq-1)}) 
&
\text{if}\ 
\Gamma_{SG,1}(N,p,q)=0>F_{SG,2}(N,p,q),
\\
\exp(C\ep^{-p(pq-1)}) 
&
\text{if}\ 
\Gamma_{SG,1}(N,p,q)<0=F_{SG,2}(N,p,q),
\\
\exp(C\ep^{-(pq-1)}) 
&
\text{if}\ 
\Gamma_{SG,1}(N,p,q)=0=F_{SG,2}(N,p,q)
\end{cases}
\]
for every 
$\ep\in (0,\ep_0]$, where $\ep_0$ and $C$ are positive constants independent of $\ep$.
\end{proposition}
\begin{remark}
On the critical curve, we could find
lifespan estimates 
including exponential functions. 
At the intersection point of two critical curves $\{\Gamma_{SG,1}=0\}$ 
and $\{\Gamma_{SG,2}=0\}$, some discontinuity in the sense 
of lifespan estimates appears. 
\end{remark}

\begin{proof}
{\bf (The case $\Gamma_{SG}(N,p,q)>0$)}. 
By 
Lemma \ref{lem:concentration} for $v$ and 
Lemma \ref{lem:concentration2} for $\pa_tu$, we have 
\begin{align}
\label{eq:SG:low1}
\delta_1
\Big(
I[g_1] \ep
\Big)^p
R^{N-\frac{N-1}{2}p}
&\leq 
\int_0^T\int_{\R^N}
|\pa_tu|^p\psi_R^*
\,dx\,dt, 
\\
\label{eq:SG:low2}
\delta_1'
\Big(
I[g_2] \ep 
\Big)^q
R^{N-\frac{N-1}{2}q}
&\leq 
\int_0^T\int_{\R^N}
|v|^q\psi_R^*
\,dx\,dt.
\end{align}
On the other hand, 
since 
$u$ is a super-solution of $\pa_t^2u-\Delta u =G= |v|^q$
and 
$v$ is a super-solution of $\pa_t^2v-\Delta v =\widetilde{G}=|\pa_tu|^p$
using Lemma \ref{lem:key0} {\bf (ii)} 
with $u$, we have
\begin{align}
\nonumber
\int_0^T\int_{\R^N}
|v|^q\psi_R
\,dx\,dt
&\leq 
CR^{-1}
\int_0^T\int_{\R^N}
|\pa_tu|[\psi_R^*]^{\frac{1}{p}}
\,dx\,dt
\\
\label{eq:SG:nonli1}
&\leq 
CR^{\frac{-1+N(p-1)}{p}}
\left(
\int_0^T\int_{\R^N}
|\pa_tu|^p\psi_R^*
\,dx\,dt
\right)^{\frac{1}{p}}
\end{align}
and using Lemma \ref{lem:key0} {\bf (i)} for $v$, we have
\begin{align}
\nonumber
\int_0^T\int_{\R^N}
|\pa_tu|^p\psi_R
\,dx\,dt
&\leq 
CR^{-2}
\int_0^T\int_{\R^N}
|v|[\psi_R^*]^{\frac{1}{q}}
\,dx\,dt
\\
\label{eq:SG:nonli2}
&\leq 
CR^{\frac{-2+(N-1)(q-1)}{q}}
\left(
\int_0^T\int_{\R^N}
|v|^q\psi_R^*
\,dx\,dt
\right)^{\frac{1}{q}}.
\end{align}
Combining the above inequalities, we deduce
\begin{align*}
\left(
\int_0^T\int_{\R^N}
|\pa_tu|^p\psi_R
\,dx\,dt
\right)^{pq}
&\leq CR^{[-2+(N-1)(q-1)]p-1+N(p-1)}
\int_0^T\int_{\R^N}
|\pa_tu|^p\psi_R^*
\,dx\,dt
\\
&\leq CR^{-pq-p-1+N(pq-1)}
\int_0^T\int_{\R^N}
|\pa_tu|^p\psi_R^*
\,dx\,dt
\end{align*}
and 
\begin{align*}
\left(
\int_0^T\int_{\R^N}
|v|^q\psi_R
\,dx\,dt
\right)^{pq}
&\leq CR^{[-1+N(p-1)]q-2+(N-1)(q-1)}
\int_0^T\int_{\R^N}
|v|^q\psi_R^*
\,dx\,dt
\\
&\leq CR^{-2q-1+N(pq-1)}
\int_0^T\int_{\R^N}
|v|^q\psi_R^*
\,dx\,dt.
\end{align*}
These yield that 
\begin{align*}
\int_0^T\int_{\R^N}
|\pa_tu|^p\psi_R
\,dx\,dt
&
\leq CR^{N-\frac{pq+p+1}{pq-1}}
\\
\int_0^T\int_{\R^N}
|v|^q\psi_R
\,dx\,dt
&
\leq CR^{N-\frac{2q+1}{pq-1}},
\end{align*}
and therefore combining \eqref{eq:SG:low1} and \eqref{eq:SG:low2}, 
we obtain the desired estimates for $T$ for $\Gamma_{SG}(N,p,q)>0$. 

{\bf (The case $F_{SG,1}(N,p,q)=0>F_{SG,2}(N,p,q)$)}. 
Observe that the condition $F_{SG,1}(N,p,q)=0$ yields
\begin{align}
\left(N-\frac{N-1}{2}p-\beta_q\right)q +
\left(N-\frac{N-1}{2}q-\beta_p-1\right)
=-F_{SG,1}(N,p,q)=0.
\label{eq:SG-crit1}
\end{align}
We see by \eqref{eq:SG:low1} and \eqref{eq:SG:nonli2} that
\begin{align*}
\int_0^T\int_{\R^N}
|v|^q\psi_R^*
\,dx\,dt
&\geq
C^{-q}R^{2-(N-1)(q-1)}
\left(\int_0^T\int_{\R^N}
|\pa_tu|^p\psi_R
\,dx\,dt
\right)^q
\\
&
\geq 
C^{-q}\delta_1^q
\Big(
I[g_1] 
\ep
\Big)^{pq}
R^{(N-\frac{N-1}{2}p)q+2-(N-1)(q-1)}
\\
&
\geq 
C^{-q}\delta_1^q
\Big(
I[g_1] 
\ep
\Big)^{pq}
R^{(N-\frac{N-1}{2}p-\beta_q)q+N-\frac{N-1}{2}q}
\\
&
=
C^{-q}\delta_1^q
\Big(
I[g_1]
\ep
\Big)^{pq}
R^{\beta_p+1}
\end{align*}
and therefore
\[
\int_0^T\int_{\R^N}
|v|^q\Phi_{\beta,\lambda}\psi_R^*
\,dx\,dt\geq C^{-q}\delta_1^q
\Big(
\ep
I[g_1]
\Big)^{pq}
\]
with $\beta=\beta_p+1$ and $\lambda=\lambda_{\beta_p+1}$. 
On the other hand, using Lemma \ref{lem:key} {\bf (ii)} with $\beta=\beta_p+1$, 
Lemma \ref{lem:int-phi}, \eqref{eq:SG:nonli2}
and the condition $F_{SG,1}(N,p,q)=0$ again, we deduce
\begin{align*}
\left(
\int_0^T\int_{\R^N}
|v|^q\Phi_{\beta,\lambda}\psi_R
\,dx\,dt
\right)^{pq}
&\leq 
\left(
\frac{C}{R}
\int_0^T\int_{\R^N}
|\pa_tu|\Phi_{\beta,\lambda}[\psi_R]^{\frac{1}{p}}
\,dx\,dt
\right)^{pq}
\\
&\leq 
CR^{-(N-\frac{N-1}{2}p)q}(\log R)^{q(p-1)}
\left(
\int_0^T\int_{\R^N}
|\pa_tu|^p\psi_R^*
\,dx\,dt
\right)^{q}
\\
&\leq 
CR^{-(\beta_p+1)}(\log R)^{q(p-1)}
\int_0^T\int_{\R^N}
|v|^q\psi_R^*
\,dx\,dt
\\
&\leq 
C(\log R)^{q(p-1)}
\int_0^T\int_{\R^N}
|v|^q\Phi_{\beta,\lambda}\psi_R^*
\,dx\,dt.
\end{align*}
Lemmas \ref{lem:int-to-diff} and \ref{lem:lifespan-crit} 
with 
$\delta=\ep^{pq}$, $p_1=pq$ and $p_2=q(p-1)+1$
imply $T\leq \exp(C\ep^{-p(pq-1)})$.

{\bf (The case $F_{SG,1}(N,p,q)<0=F_{SG,2}(N,p,q)$)}. 
Observe that the condition $F_{SG,2}(N,p,q)=0$ yields
\begin{align}
\left(N-\frac{N-1}{2}p-\beta_q\right) +
\left(N-\frac{N-1}{2}q-\beta_p-1\right)p
=-F_{SG,2}(N,p,q)=0.
\label{eq:SG-crit2}
\end{align}
We see by \eqref{eq:SG:low2} and \eqref{eq:SG:nonli1} that
\begin{align*}
\int_0^T\int_{\R^N}
|\pa_tu|^p\psi_R^*
\,dx\,dt
&\geq
C^{-p}R^{1-N(p-1)}
\left(\int_0^T\int_{\R^N}
|v|^q\psi_R
\,dx\,dt
\right)^p
\\
&
\geq 
C^{-p}\delta_2^p
\Big(
I[g_2] 
\ep
\Big)^{pq}
R^{(N-\frac{N-1}{2}q)p+1-N(p-1)}
\\
&
=
C^{-p}\delta_2^p
\Big(
I[g_2] 
\ep
\Big)^{pq}
R^{\beta_q}
\end{align*}
and therefore
\[
\int_0^T\int_{\R^N}
|\pa_tu|^p\Phi_{\beta,\lambda}\psi_R^*
\,dx\,dt\geq C^{-p}\delta_2^p
\Big(
I[g_2] 
\ep
\Big)^{pq}
\]
with $\beta=\beta_q$ and $\lambda=\lambda_{\beta_q}$. 
On the other hand, using Lemma \ref{lem:key} {\bf (i)} with $\beta=\beta_q$, 
Lemma \ref{lem:int-phi}, 
\eqref{eq:SG:nonli2} and the condition $F_{SG,2}(N,p,q)=0$, 
we have 
\begin{align}
\nonumber 
\left(
\int_0^T\int_{\R^N}
|\pa_tu|^p\Phi_{\beta,\lambda}\psi_R
\,dx\,dt
\right)^{pq}
&\leq 
\left(
\frac{C}{R}
\int_0^T\int_{\R^N}
|v|\Phi_{\beta+1,\lambda}[\psi_R^*]^{\frac{1}{q}}
\,dx\,dt
\right)^{pq}
\\
\nonumber 
&\leq 
CR^{-(N-\frac{N-1}{2}q)p}(\log R)^{p(q-1)}
\left(
\int_0^T\int_{\R^N}
|v|^q\psi_R^*
\,dx\,dt
\right)^{p}
\\
\nonumber 
&\leq 
CR^{-\beta_q}(\log R)^{p(q-1)}
\int_0^T\int_{\R^N}
|\pa_tu|^p\psi_R^*
\,dx\,dt
\\
\label{eq:SG3}
&\leq 
C(\log R)^{p(q-1)}
\int_0^T\int_{\R^N}
|\pa_tu|^p\Phi_{\beta,\lambda}\psi_R^*
\,dx\,dt.
\end{align}
Lemmas \ref{lem:int-to-diff} and \ref{lem:lifespan-crit} 
with 
$\delta=\ep^{pq}$, $p_1=pq$ and $p_2=p(q-1)+1$
imply $T\leq \exp(C\ep^{-q(pq-1)})$. 

{\bf (The case $F_{SG,1}(N,p,q)=F_{SG,2}(N,p,q)=0$)}. 
In this case we see from \eqref{eq:SG-crit1} and  \eqref{eq:SG-crit2} that 
\[
N-\frac{N-1}{2}p=\beta_q, 
\quad 
N-\frac{N-1}{2}q=\beta_p+1.
\]
This implies that \eqref{eq:SG:low1} can be rewritten as 
\begin{align*}
\delta_1
\Big(
I[g_1] \ep
\Big)^p
&\leq 
R^{-N+\frac{N-1}{2}p}
\int_0^T\int_{\R^N}
|\pa_tu|^p\psi_R^*
\,dx\,dt
\\
&\leq 
C\int_0^T\int_{\R^N}
|\pa_tu|^p\Phi_{\beta,\lambda}\psi_R^*
\,dx\,dt
\end{align*}
with $\beta=\beta_q$ and $\lambda=\lambda_{\beta_q}$. 
In view of the above estimate and \eqref{eq:SG3}, 
Lemma \ref{lem:int-to-diff} with $w=|\pa_tu|^p\Phi_{\beta,\lambda}$ 
and Lemma \ref{lem:lifespan-crit} 
with $\delta=\ep^{p}$, 
$p_1=pq$ and $p_2=p(q-1)+1$
imply $T\leq \exp(C\ep^{-(pq-1)})$.
\end{proof}


\newpage



\end{document}